\documentclass[11pt]{article}
\usepackage[pagebackref,colorlinks=true,urlcolor=red,linkcolor=blue,citecolor=red]{hyperref}

\usepackage{amsthm,amssymb,amsfonts, amsmath,amsrefs,amscd}
\usepackage{color}
\usepackage[]{fancyhdr}
\usepackage[]{graphicx,wrapfig,lipsum}
\usepackage{soul}

\usepackage{enumitem}
\setlist[enumerate,1]{label=(\alph*), ref=(\alph*)}

\graphicspath{{/EPSF/}{../figures/}{figures/}}
\usepackage{mathtools}  
\mathtoolsset{showonlyrefs}
\usepackage{enumitem}
\usepackage{graphicx} 
\newcommand{\abs}[1]{\lvert #1 \rvert}
\newcommand{\srn}{\mathcal{S}'(\Rn)}
\usepackage[dvipsnames]{xcolor}

\newcommand{\sn}{\mathbb{S}^{n-1}}

\newcommand{\I}{\mathrm{i}}
\newcommand{\D}{\mathrm{d}}

\newcommand{\wh}{\widehat}
\newcommand{\vf}{\varphi}

\newcommand{\PD}{\partial}

\newcommand{\Fc}{\mathcal{F}}

\newcommand{\Sc}{\mathcal{S}}

\newcommand{\Rb}{\mathbb{R}}
\newcommand{\Sb}{\mathbb{S}}

\usepackage{amsthm}
\usepackage{amsfonts}
\newcommand{\p}{\partial}

\newcommand{\Beq}{\begin{equation}}
\newcommand{\Eeq}{\end{equation}}
\newcommand{\beq}{\begin{equation*}}
\newcommand{\eeq}{\end{equation*}}
\newcommand{\bal}{\begin{align}}
\newcommand{\eal}{\end{align}}

\newcommand*\Rn{\mathbb{R}^n}

\newcommand{\inner}[1]{\langle #1 \rangle}

\newcommand*{\Rt}{\mathbb{R}^2}

\newcommand{\sch}{\mathcal{S}}

\renewcommand{\d}{\mathrm{d}}

\renewcommand{\l}{\langle}
\renewcommand{\r}{\rangle}

\newtheorem{theorem}{Theorem}[]
\newtheorem{lemma}[theorem]{Lemma}
\newtheorem{proposition}[theorem]{Proposition}
\newcommand{\qum}[1]{\quad\mbox{#1}}
\newtheorem{corollary}[theorem]{Corollary}
\newtheorem{definition}[theorem]{Definition}
\newtheorem{remark}[theorem]{Remark}

\newcommand{\lr}[1]{\left(#1\right)}

\author{Antti Kykk\"anen\thanks{Department of Computational Applied Mathematics and Operations Research, Rice University, Houston, TX, USA. \url{ak272@rice.edu}}\and Rohit Kumar Mishra\thanks{Department of Mathematics, Indian Institute of Technology, Gandhinagar, Gujarat, India. \url{rohit.m@iitgn.ac.in}, \url{rohittifr2011@gmail.com}}
\and Suman Kumar Sahoo\thanks{Department of Mathematics, Indian Institute of Technology, Bombay, Maharashtra, India. \url{suman@math.iitb.ac.in }}}

\title{Generalized Helmholtz-type decompositions of symmetric tensor fields and applications to ray transforms}

\begin{document}
\maketitle
\begin{abstract}
%In this work, we study a solenoidal-potential type decomposition of a symmetric $m$-tensor field in $\Rb^2$, and its implications in addressing injectivity questions for a set of integral transforms in the field of integral geometry.
We study a solenoidal-potential type decomposition of a symmetric $m$-tensor field in $\Rb^2$, and its implications to injectivity questions for the momentum and elastic ray transforms. For symmetric tensor fields, a general decomposition with a restriction on the dimension and order of the decomposition was proved in~\cite{Rohit_Suman}. We extend the result to dimension $2$ under a mean-zero assumption. We use the decomposition in $2$ dimensions to prove the injectivity of the momentum and elastic ray transforms. We also prove a connection between the two integral transforms for $2$-tensors.
%A general decomposition result was proved in the article \textit{``Injectivity and range description of integral moment transforms over $m$-tensor fields in $\Rb^n$, SIAM J. Math. Anal. 53 (2021), no. 1, 253–278. MR4198570"}, with a restriction on the dimension and order of the decomposition. With a mean-zero assumption on the tensor field, we prove the decomposition for a symmetric $m$-tensor field in $\Rb^2$, which was not possible in the previous work. 
Later, we use our decomposition to prove the injectivity of integral transforms, including longitudinal and elastic ray transforms, without any mean-zero assumption on tensor fields. Next, we explore connections between different integral transforms and use these to relate their corresponding properties. 
\end{abstract}
\vspace{2mm} 
\vspace{2mm} 

\section{Introduction}

The well-known Helmholtz Decomposition states that a ``sufficiently smooth, rapidly decaying" vector field in $\mathbb{R}^3$ can be uniquely decomposed into a curl-free and a divergence-free part. The curl-free and divergence-free parts are often referred to as the potential part and solenoidal part, respectively. Generalizing this decomposition, Sharafutdinov~\cite{Sharafutdinov_book} proved that a symmetric $m$-tensor field in $\Rb^n$ can be uniquely decomposed into a solenoidal part and a potential part of the given tensor field. Recently in~\cite{Rohit_Suman}, the authors introduced $k$-solenoidal and $k$-potential tensor fields as a generalization of solenoidal and potential tensor fields in $\Rb^n$.
% A more general decomposition of symmetric $m$-tensor fields into $k$-solenoidal and $k$-potential tensor fields appeared in~\cite{Rohit_Suman}.
Their decomposition is proven under a restriction on the dimension and the order of tensor fields, which makes the decomposition not applicable for $n=2$. Improving on this, we prove that the decomposition result holds in $2$ dimensions as well, under a mean-zero condition on the tensor field. Furthermore, we prove that the imposed condition cannot be removed, as we also prove that existence of such a decomposition implies the mean-zero condition. There are other kinds of decompositions studied in the literature, specifically for problems in integral geometry; we refer the reader to~\cite{derevtsov3, Polyakova_Svetov_normal_Radon_numerical, ilmavirta2025elasticraytransform} and the references therein. 

Such decompositions play a fundamental role in many applications. In fluid dynamics, the Helmholtz decomposition allows one to split a velocity field into its divergence-free (solenoidal) part and its irrotational (gradient) part. This separation is essential in the analysis of incompressible flows, where the pressure term acts as a Lagrange multiplier enforcing the divergence-free constraint. In particular, by projecting the Navier–Stokes equations onto the space of divergence-free vector fields via the Helmholtz decomposition, one eliminates the pressure term and obtains the Stokes operator as the linear part of the dynamics. This viewpoint is fundamental in the linearization of the Navier–Stokes equations around equilibrium states and in the study of well-posedness and stability, see for instance~\cite{Chorin,Galdi}. The Helmholtz decomposition also plays an important role in electrodynamics and medical imaging. More generally, such decompositions provide powerful tools for characterizing the kernels of integral transforms and for establishing injectivity modulo the identified kernel. In our work, we apply this framework to momentum ray transforms and to elastic ray transforms of tensor fields.
%Such decompositions play a fundamental role in many applications. In particular, the Helmholtz decomposition of vector fields is central to fluid dynamics, where it separates a velocity field into its divergence-free (incompressible/solenoidal) component and its irrotational (gradient) component. The Helmholtz decomposition is applied to linearized the incompressible Navier Stokes equation and as a result one obtains Stokes equation~\cite{??}\antti{Citation missing}. The decomposition is also fundamental in electrodynamics and medical imaging. More generally, such decompositions provide powerful tools for characterizing the kernels of integral transforms and for establishing injectivity modulo the identified kernel. In our work, we apply this framework to momentum ray transforms and to elastic ray transforms of tensor fields.

% Such characterization result of vector fields and tensor fields are very useful in many applications. The Helmholtz decomposition  of vector fields has application in fluid dynamics and elElectrodynamics, and medical imaging. 
% Such decompositions are very useful for characterizing the kernels of integral transforms and for proving injectivity modulo the identified kernel. We have applied our decomposition to the momentum ray transforms and to the elastic ray transforms of tensor fields.

For momentum ray transforms, we show that the first two integral moments of a symmetric $2$-tensor field are zero if and only if the tensor field is the Hessian of a scalar, real-valued function. To prove injectivity on solenoidal tensors, we do not require the mean-zero condition (which was essential for the decomposition), since it is automatically satisfied if the momentum transforms are zero. Momentum ray transform appears to study Calder\'on type inverse problems for biharmonic and polyharmonic operators, see for instance~\cite{BKS_Siam,SS_jde,BKSU_nonlinear_cpde}. 
Such injectivity results are studied by many authors in different settings; the author may refer to~\cite{Sharafutdinov_1986_momentum,Rohit_Suman,Kamran_attenuated_mrt,Plamen_Venky,KMSS,JMKS_normal,AKS_ucp,Monard_Efficient_TT,Derevtsov_2023,derevtsov3,Gunther_2013,Hoop_2019,Gunther_Book,Gunther_Mikko} and some references therein.

Elastic ray transforms arise as linearizations elastic travel times around a homogeneous and isotropic background (see~\cite{ilmavirta2025elasticraytransform} for details). A solenoidal-potential decomposition adapted to the elastic ray transform setting appeared in~\cite{ilmavirta2025elasticraytransform} under the dimensional restriction $n \geq 3$.
We extend their result to $2$-dimensional domains. For elastic tensor fields the potential part of the tensor field consists of two not mutually disjoint parts (see Section~\ref{sec:Def and notation} for definitions of the operators, c.f.~\cite{ilmavirta2025elasticraytransform}). Additionally, we use our decomposition to extend the injectivity result from~\cite{ilmavirta2025elasticraytransform} into the $2$-dimensional setting.

Finally, we conclude the article by relating the elastic ray transform, the momentum ray transforms, and the mixed ray transforms and use the observed relations to characterize their kernels in the case of $2$-tensors.

\subsection{Organization of the article}

In Section~\ref{sec:Def and notation}, we introduce the various notations and definitions that will be used in the article. Section~\ref{sec:Decomposition Results} focuses on the $k$-solenoidal and $k$-potential decomposition result. Section~\ref{sec:elastic-decomp} focus on solenoidal-potential decomposition of elastic tensor fields. In Sections~\ref{sec:Application to moments} and~\ref{sec:Application to elastic ray}, we apply our decompositions to the momentum ray transform and the elastic ray transforms of symmetric tensor fields in $\Rb^2$, respectively. The relation between different integral transforms is discussed in Section~\ref{sec:relation between integral transform}. In appendix~\ref{appen_homo_dis}, we recall details on homogenous distributions which are used throughout the article. Finally, we complete the paper with an acknowledgment in Section~\ref{sec:acknowledgement}.

\section{Definitions and notation}\label{sec:Def and notation}
In this section, we introduce some necessary definitions and notations used throughout this article. Most of these definitions and notations can be found in the book~\cite{Sharafutdinov_book} by Sharafutdinov. We also refer to the monograph~\cite{Gunther_Book} for a detailed study of integral transforms from a geometric viewpoint.

Let $T^m(\Rb^n)$ denotes the space of $m$-tensors on $\Rb^n$ and $S^m = S^m(\Rb^n)$ be its subset containing symmetric $m$-tensors.  
% The natural projection of $T^m(\Rb^n)$ onto the space of symmetric tensors $S^m(\Rb^n)$, $\sigma : T^m(\Rb^n) \rightarrow S^m(\Rb^n)$, is given by  
% \begin{align}\label{eq:definition of sigma}
% (\sigma v)_{i_1\dots i_m} = \frac{1}{m!}\sum_{\pi \in \Pi_m} v_{\pi(i_1)\dots \pi(i_m)} 
% \end{align}
% where $\Pi_m$ is the set of permutations of order $m$.
% \vspace{2mm}
We use the notations $\otimes$ to denote the tensor product of tensors/vectors. More explicitly, for $v\in T^m(\mathbb{R}^n)$ and $u\in T^{m'}(\mathbb{R}^n)$, the tensor product of $u$ and $v$ gives an element in $T^{m+m^\prime}(\mathbb{R}^n)$, which is given by
% We use the notations $\otimes$ and $\odot$ to denote the tensor product and symmetric tensor product of tensors/vector spaces, respectively. Explicitly, for $v\in S^m(\mathbb{R}^n)$ and $u\in S^{m'}(\mathbb{R}^n)$, these tensor products are defined as follows: 
\begin{align*}
    (v \otimes u)_{i_1\dots i_{m+m'}} &:=  v_{i_1\dots i_m}u_{i_{m+1}\dots i_{m+m'}}.
   % v\odot u &= \frac{1}{2}\left(v \otimes u +u \otimes v\right).
%    (v\odot u)_{i_1\dots i_{m+m'}} &:= \tblue{\frac{m!m'!}{(m+m')!}\sum_{\sigma \in \Pi_{m}, \tilde{\sigma} \in \Pi_{m'}}v_{\sigma(1)\sigma(2)\dots\sigma(m)}u_{\tilde{\sigma}(1)\tilde{\sigma}(2)\dots\tilde{\sigma}(m')}},
    %\tblue{\frac{1}{(k+\ell)!}\sum_{\sigma \in \Pi_{k+\ell}}v_{\sigma(1)\sigma(2)\dots\sigma(m)}u_{\sigma(m+1)\sigma(m+2)\dots\sigma(m+m')}},
    %\left(v \otimes u +u \otimes v\right).
\end{align*}
% \tblue{where $\Pi_m$ and $\Pi_{m'}$ are the set of permutation of order $m$ and $m'$, respectively.}
% \vspace{2mm}\\
Let $E^m(n)$ denotes the space obtained by taking symmetric tensor product of $\mathit{S}^m(\mathbb{R}^n)$ with itself. This means $E^m(n)$ is a collection of $2m$-tensors that are symmetric in the first $m$ and the last $m$ indices, and it is also unchanged if we swap the first $m$ and the last $m$ components.
% For non-integers $k$ and $\ell$, let $\mathit{S}^{k,\ell}(\mathbb{R}^n):= \mathit{S}^{k}(\mathbb{R}^n)\odot \mathit{S}^{\ell}(\mathbb{R}^n)$ 
% be the space of $(k+\ell)$ tensors that are symmetric with respect to the first $k$ and last $\ell$ indices in $\FR^n$. In the particular case when $k = \ell$, we have block symmetry as well, that is, the tensors in this space are also unchanged if we swap the first $k$ and last $k$ indices. We have used a different notation for this special case.
% \vspace{2mm}\\
% Also, we will use the following conventions while working with these spaces:
% \begin{align*}
%     \mathit{S}^{m,0}(\mathbb{R}^n) &=  \mathit{S}^{0,m}(\mathbb{R}^n) = \mathit{S}^{m}(\mathbb{R}^n)\\
%      \mathit{S}^{m,m}(\mathbb{R}^n) &= \mathit{E}^{m}(n), \mbox{ this notation is used in \cite{ilmavirta2025elasticraytransform}}.
% \end{align*}
The notation $E^m(n)$ is used in~\cite{ilmavirta2025elasticraytransform}, and the tensors in the space $E^m(n)$ are often referred to as \textit{elastic tensors.} We will also use the same notation $E^m(n)$ when we work with elastic ray transform. Note that an elastic $m$ tensor is actually a $2m$ tensor having specific symmetries. Please note that $\mathit{E}^{1}(n) = S^2(\mathbb{R}^n)$. The following inclusion may be observed directly:
$$ S^{2m}(\mathbb{R}^n) \subsetneq E^{m}(n)  \subsetneq T^{2m}(\mathbb{R}^n).$$
The natural projection of $T^{2m}(\Rb^n)$ onto the space $S^m(\Rb^n)\otimes S^m(\Rb^n)$ is denoted by $\sigma^{m,m} : T^{2m}(\Rb^n) \rightarrow S^m(\Rb^n)\otimes  S^m(\Rb^n)$ and is given by  
\begin{align}\label{eq:definition of sigma(k,l)}
(\sigma^{m,m} v)_{i_1\dots i_mj_1\dots j_m} = \frac{1}{(m!)^2}\sum_{\pi_1, \pi_2 \in \Pi_m}  v_{\pi_1(i_1)\dots \pi_1(i_m)\pi_2(j_1)\dots \pi_2(j_m)},
\end{align}
where $\Pi_m$ is the set of permutations of order $m$. The projection $\varepsilon : T^{2m}(\Rb^n) \rightarrow E^m(n)$ is defined as follows:
\begin{align}\label{eq:definition of epsilon}
    (\varepsilon v)_{i_1\dots i_mj_1\dots j_m} = \frac{1}{2}\left((\sigma^{m,m} v)_{i_1\dots i_mj_1\dots j_m} +(\sigma^{m,m} v)_{j_1\dots j_m i_1\dots i_m}\right).
\end{align}
Further, let  $C^\infty(\mathbb{R}^n;\mathit{S}^{m})$ denote the space of smooth symmetric $m$-tensor fields and $C_c^\infty(\Rb^n;\mathit{S}^{m})$ be the space of compactly supported tensor fields in $C^\infty (\Rb^n;\mathit{S}^{m})$. In a similar spirit, tensor fields defined on other function spaces will be used throughout the article. For instance, $\Sc(\Rb^n;\mathit{S}^{m})$ denotes the space of symmetric $m$-tensor fields with components in the Schwartz space $\mathcal{S}(\Rn)$. The same convention will be used for tensor fields in $E^m(n)$. The tensor fields in the $\mathcal{S}(\Rb^n; E^m(n))$ are usually known as \textit{elastic tensor fields}. In standard Cartesian coordinates, any element $f \in \mathcal{S}(\Rb^n;\mathit{S}^{m})$ can be written as
\[ f(x) = \sum_{i_1, i_2, \dots, i_m = 1}^nf_{i_1\dots i_m}(x) dx^{i_1} \dots dx^{i_m} \]
%$$ f(x) = f_{i_1\dots i_m}(x) dx^{i_1} \dots dx^{i_m}$$
where  $f_{i_1 \dots i_m} \in \Sc(\Rb^n)$ are symmetric in all its indices. 
%For repeated indices, the Einstein summation convention will be assumed throughout this article. 
We will not distinguish between covariant and contravariant tensors as we are working with the Euclidean metric.  \vspace{2mm}\\
For $x \in \Rb^n$,  we define the \textit{symmetric multiplication operator} $i_x :S^m(\Rb^n)  \rightarrow S^{m+1}(\Rb^n)$ by 
$$(i_x f)_{i_1i_2\dots i_{m+1}} =\sigma(i_1, \dots, i_m, i_{m+1})(x_{i_{m+1}}f_{i_{1}i_{2}\dots i_{m}}).$$
In the same spirit, we also define the dual of $i_x$, \textit{the convolution operator}, $j_x :S^m(\Rb^n)  \rightarrow S^{m-1}(\Rb^n)$ by 
$$(j_x f)_{i_1i_2\dots i_{m-1}} = \sum_{i_m=1}^n f_{i_1i_2\dots i_m}x_{i_m}.$$ 
The composition of these operators will be essential in the next section to prove the decomposition theorem and hence for the convenience of the reader, we introduce the operators $i^k_{x}:  S^m(\Rb^n) \rightarrow S^{m+k}(\Rb^n)$ and $j^k_{x}:  S^{m+k}(\Rb^n) \rightarrow S^{m}(\Rb^n)$, for any fixed integer $k \geq 1$, as follows (cf.~\cite[Page 4]{Rohit_Suman}):
\begin{align*}
\left(i^k_{x}f\right)_{i_1i_2\dots i_{m+k}} &= \underbrace{i_x \circ i_x\circ \cdots\circ i_{x}}_{k \mbox{ times}} f=\sigma(i_1, \dots, i_m, i_{m+1}, \dots, i_{m+k})(x_{i_{m+1}}\dots x_{i_{m+k}}f_{i_{1}i_{2}\dots i_{m}})\\
\left(j^k_{x}f\right)_{i_1i_2\dots i_{m}}&=\underbrace{j_x \circ j_x\circ \cdots\circ j_{x}}_{k \mbox{ times}} f= \sum_{i_{m+1}, \dots,i_{m+k} = 1 }^nx_{i_{m+1}}\dots x_{i_{m+k}}f_{i_{1}i_{2}\dots i_{m}i_{m+1}\dots i_{m+k} }.
\end{align*}
Let $\tilde{\D}: C^\infty(\Rb^n;T^{m})\rightarrow C^\infty(\Rb^n;T^{m+1})$ be a differential operator defined as
$$(\tilde{\D} u)_{i_1\dots i_mi_{m+1}} = \frac{\partial u_{i_1\dots i_m} }{\partial x_{i_{m+1}}}.$$
The operator of \textit{inner differentiation} or \textit{symmetrized derivative} $\D:C^\infty(\Rb^n;S^m)\rightarrow C^\infty(\Rb^n;S^{m+1})$ is defined by
$$(\D u)_{i_1\dots i_mi_{m+1}} = \sigma(i_1, \dots , i_{m+1}) \left(\frac{\partial u_{i_1\dots i_m} }{\partial x_{i_{m+1}}}\right),$$
where $\sigma$ is defined in equation~\eqref{eq:definition of sigma(k,l)}. The \textit{divergence} operator $\delta:C^\infty(\Rb^n;S^{m})\rightarrow C^\infty(\Rb^n;S^{m-1})$ is defined by the formula $$ (\delta u)_{i_1\dots i_{m-1}} = \sum_{j=1}^n \frac{\partial u_{i_1\dots i_{m-1}j} }{\partial x_{j}}.$$
The following second-order differential operator will also be used in the later sections of the manuscript.\\
For the purposes of the elastic ray transform appearing later, we also define the operator $H: C^\infty(\Rb^n;E^m(n))\rightarrow C^\infty(\Rb^n;E^{m+1}(n))$ is defined as
\begin{align}
    (H v)_{i_1\dots i_mi_{m+1}j_1\dots j_mj_{m+1}} =  \varepsilon \left(\frac{\partial^2}{\partial x_{i_{m+1}}\partial x_{j_{m+1}}}v_{i_1\dots i_mj_1\dots j_m}\right). 
\end{align}
The formal $L^2$-adjoint $H^*: C^\infty(\Rb^n;E^{m+1}(n))\rightarrow C^\infty(\Rb^n;E^{m}(n))$ of $H$ is given by
\begin{align}
    (H^* v)_{i_1\dots i_mj_1\dots j_m} = \frac{\partial^2}{\partial x_{i_{m+1}}\partial x_{j_{m+1}}}v_{i_1\dots i_mi_{m+1}j_1\dots j_mj_{m+1}}.
\end{align}
In addition we shall make use of the following operator, specific to elastic $2$-tensor fields. We define the operator $K \colon \mathcal{S}(\mathbb{R}^n;\Rb^n) \to \mathcal{S}(\mathbb{R}^n;E^2(n))$ by $Ku =\varepsilon(\D u \otimes I)$ where $I$ is the identity matrix. In coordinates, we have
\begin{equation}
(Ku)_{ijkl} =
\frac14
\left(
\frac{\partial u_j}{\partial x_i}
+
\frac{\partial u_i}{\partial x_j}
\right)
\delta_{kl}
+
\frac14
\left(
\frac{\partial u_k}{\partial x_l}
+
\frac{\partial u_l}{\partial x_k}
\right)
\delta_{ij}.
\end{equation}
The formal $L^2$-adjoint of $K$ is the operator $K^* \colon \mathcal{S}(\mathbb{R}^n;E^2(n)) \to \mathcal{S}(\mathbb{R}^n;\Rb^n)$ defined by
\begin{equation}
(K^*w)_i
=
-
\sum_{j,k}
\frac{\partial w_{ijkk}}{\partial x_j}.
\end{equation}
% Next, we define two special tensor fields, generalizations of the solenoidal and potential tensor fields, respectively (see \cite[Definition 3.1]{Rohit_Suman}). \antti{Add here defition for solenoidal elastic tensor fields? Uniformize. I'll do this later, when I restructure the article a bit.}
% \begin{definition}[$k$-solenoidal and $k$-potential tensor fields]\label{def:k-potential and k-solenoidal}
% For any fixed $1 \leq k \leq m$, a symmetric $m$-tensor field $f \in C^\infty(\Rb^n;S^m)$ is said to be 
% \begin{enumerate}
%     \item $k$-solenoidal tensor field if   $\delta^k f = 0.$
% %     \begin{align*}
% %     \delta^k f = 0.
% % \end{align*}
% \item $k$-potential tensor field if there exists a $(m-k)$-tensor field $v \in C^\infty(\Rb^n, S^{m-k})$ such that   $ f = \D^k v.$  %\begin{align*}
% %   f = \D^k v.
% % \end{align*}
% \end{enumerate}
% \end{definition}
% For $k=1$, the $k$-solenoidal and $k$-potential tensor fields coincide with the usual solenoidal and potential tensor fields, respectively. 
The Fourier transform  $\Fc :\Sc(T\Sb^1) \longrightarrow \Sc(T\Sb^1)$ is defined as follows, see~\cite[Section 2.1]{Sharafutdinov_book}:
\begin{align}\label{eq:Fourier transform on sphere bundle}
    \Fc (\vf) (y, \xi) =    \widehat{\vf}(y, \xi) = \frac{1}{(2 \pi)^{1/2}}\int_{\xi^\perp} e^{-i x\cdot y} \vf(x, \xi)\, dx,
\end{align}
where $dx$ is the arc-length measure on the line $\xi^\perp = \{ x\in \Rb^2 : \l x,\xi \r = 0\}.$

\section{Decomposition results}
In this section, we provide two decomposition theorems in $\Rt$. We begin by proving the decomposition of symmetric rank-two tensor fields, followed by the analogous result for elastic rank-two tensors. 
\subsection{A decomposition result for symmetric $2$-tensor fields in \texorpdfstring{$\Rb^2$}{}}\label{sec:Decomposition Results}
The goal of this section is to prove that any symmetric $2$-tensor field can be decomposed uniquely into a $2$-solenoidal part and a $2$-potential part. This decomposition theorem extends the result~\cite[Theorem 2.6.2]{Sharafutdinov_book}, which gives a unique decomposition of a symmetric $m$-tensor field into its solenoidal part and potential part. This result can be viewed as a generalization of the well-known Helmholtz (also known as the Helmholtz-Hodge) decomposition of a vector field into a divergence-free (solenoidal) part and a curl-free (potential) part.

\begin{definition}[$k$-solenoidal and $k$-potential tensor fields]\label{def:k-potential and k-solenoidal}
For any fixed $1 \leq k \leq m$, a symmetric $m$-tensor field $f \in C^\infty(\Rb^n;S^m)$ is said to be 
\begin{enumerate}
    \item $k$-solenoidal tensor field if   $\delta^k f = 0$, and
%     \begin{align*}
%     \delta^k f = 0.
% \end{align*}
\item $k$-potential tensor field if there exists a $(m-k)$-tensor field $v \in C^\infty(\Rb^n, S^{m-k})$ such that   $ f = \D^k v.$  %\begin{align*}
%   f = \D^k v.
% \end{align*}
\end{enumerate}
\end{definition}
\noindent For $k=1$, the $k$-solenoidal and $k$-potential tensor fields coincide with the usual solenoidal and potential tensor fields, respectively. The definition already appeared in~\cite[Definition 3.1]{Rohit_Suman}.

We begin by recalling a tensor decomposition previously shown in~\cite[Lemma 3.2]{Rohit_Suman}. We only recall a simplified version (corresponding to $m=2$ and $n=2$) of this result, which we will need in the following discussion. 

\begin{lemma}[{\cite[Lemma 3.2]{Rohit_Suman} and \cite[Lemma 2.6.1]{Sharafutdinov_book}}]
\label{th: decomposition of f in frequency variable} 
Given a tensor field $f\in C^{\infty}(\Rb^2;S^2)$, there exist unique $g\in C^{\infty}(\Rb^2\setminus\{0\};S^2)$ and  $v\in C^{\infty}(\Rb^2\setminus\{0\})$ such that the following decomposition of $f$ holds:
	\begin{align}\label{eq: decomposition of f in frequency variable}
	f(x) = g(x) + i^2_{x} v (x),  \qquad x \in \Rn\setminus\{0\}
	\end{align}
and $g$ satisfies $j^2_{x} g(x) =0$, $x \in \Rb^2\setminus\{0\}$. For $x \in \Rb^2\setminus\{0\}$, the fields $g$ and $v$ are explicitly expressed in terms of $f$ as follows:
\begin{align}\label{eq: expression of v and g}
v(x) 
= \sum_{i,j=1}^2\frac{x_{i} x_{j}}{|x|^{4}}f_{i j}(x) \qquad \mbox{ and } \qquad 	g_{ij}(x) 
= \sigma(i,j)\sum_{k,\ell=1}^2\left( \delta_{ik}\delta_{j\ell} - \frac{x_{i} x_{j}x_{k} x_{\ell}}{|x|^{4}}\right)f_{k \ell}(x).
\end{align}	
\end{lemma}
\noindent We are now ready to show a decomposition theorem for symmetric 2-tensor fields. To proceed, we introduce the \emph{mean-zero condition} for tensor fields. Let
$f \in \mathcal{S}(\mathbb{R}^2;S^2)$ be a symmetric $2$-tensor field. We say that $f$ satisfies the
\emph{mean-zero condition} if
$
\int_{\mathbb{R}^2} f_{ij}(x)\,dx = 0$ for  $1 \leq i,j \leq 2.$
For brevity, we shall write
\[
\int_{\mathbb{R}^2} f = 0
\]
to indicate that \(f\) satisfies the above mean-zero condition.
\vspace{2mm} \\
We are interested in decomposing a tensor field $f\in \sch(\mathbb{R}^2;S^2)$ as follows:
% \begin{align}\label{new_decomp_of_f_no_restriction}
%     f= g+\D^2 v;  \quad \ \delta^{2} g=0.
% \end{align}	
% with the property that $g(x), v(x) \rightarrow 0 \ \mbox{as}\ |x| \rightarrow \infty$ and satisfies the following improved decay condition:
% \begin{align}\label{decay_estimate}
% \abs{g}\lesssim (1+ \abs{x})^{-2}, \qum{and} \quad \abs{\D^\ell v}\lesssim (1+\abs{x})^{-\ell} , \qum{for $\ell=0,1$}.
% \end{align}
there exist uniquely determined smooth symmetric $2$-tensor field $g$ and  a scalar field $v$ such that
\begin{equation}\label{new_decomp_of_f_no_restriction}
%&\mbox{there exist uniquely determined smooth symmetric $2$-tensor field $g$ and  a scalar field $v$ such that}\\
\left\{
\begin{aligned}
&f = g+\D^2v, \quad
\delta^2 g =0, \\
&g(x),\,v(x) \to 0 \qquad \text{as } |x|\to\infty, \\
&|g(x)| \lesssim (1+|x|)^{-2},\quad
|\D^\ell v(x)| \lesssim (1+|x|)^{-\ell},
\qquad \ell=0,1.
\end{aligned}
\right.
\end{equation}
The tensor fields $g$ and $v$ will be referred to as the $2$-solenoidal and $2$-potential parts of $f$, respectively, in accordance with Definition~\ref{def:k-potential and k-solenoidal}.

\begin{theorem}
\label{th:new_decomposition}
 Let $ f \in \mathcal{S}(\mathbb{R}^2;S^2)$ be a symmetric $2$-tensor field in $\Rt$.  The following statements are equivalent.
\begin{enumerate}
        \item \label{A}  $f$ has mean-zero. 
        \item  \label{B}  $f$ satisfies the generalized solenoidal-potential decomposition stated in~\eqref{new_decomp_of_f_no_restriction}.\end{enumerate}
\end{theorem}
% \begin{theorem}\label{th:new_decomposition}
% Let $ f \in \mathcal{S}(S^2)$ be a symmetric $2$-tensor field satisfying $\displaystyle \int_{\Rt} f(x) dx = 0$. Then the decomposition \eqref{new_decomp_of_f_no_restriction} holds.
% % Then there exist uniquely determined smooth symmetric $2$-tensor field $g$ and  a scalar field $v$ satisfying    \begin{align}\label{new_decomp_of_f_no_restriction}
% %     f= g+\D^2 v;  \quad \ \delta^{2} g=0.
% % \end{align}	
% % The fields $g(x), v(x) \rightarrow 0 \ \mbox{as}\ |x| \rightarrow \infty$ with and satisfies the following improved decay condition:
% % \begin{align}
% % \abs{g}\lesssim (1+ \abs{x})^{-n}, \qum{and} \quad \abs{\D^\ell v}\lesssim (1+\abs{x})^{2-\ell-n} , \qum{for $\ell=0,1$}.
% % \end{align}
% % The tensor fields $g$ and $v$ will be referred to as the $2$-solenoidal and $2$-potential parts of $f$, respectively.\\

% \noindent Furthermore, if  the decomposition result \eqref{new_decomp_of_f_no_restriction} of $f \in\sch(\Rt)$ exists, then $f$ has mean-zero $i.e.$ $\int_{\Rt} f=0.$
% \end{theorem}

\begin{proof}
\textbf{Step 1.}
We  show that~\ref{A} implies~\ref{B}. We start by applying Lemma~\ref{th: decomposition of f in frequency variable} on $\widehat{f}$ to obtain a unique symmetric 2-tensor field $\widehat{g}$ and a scalar field  $\widehat{v}$ satisfying 
\begin{align}\label{eq:Fourier Decomposition}
\widehat{f}(y)=\widehat{g}(y) + i_{y}^{2}\widehat{v}(y) \ \ \mbox{ and }\    j_{y}^{2} \widehat{g}(y)=0.	\end{align}
Using expressions for $\widehat{g}$ and $\widehat{v}$ given in~\eqref{eq: expression of v and g}, we have that  both the fields $ \widehat{g}(y) $ and $ \widehat{v}(y) $ are smooth on $ \Rn_{0}=\Rn\setminus\{0\} $ and decay rapidly as $|y| \rightarrow \infty$. Our next step is to use our hypothesis to show $ \widehat{v}$ is integrable, which will allow us to prove the existence of $v$ by the Fourier inversion.  Recall from equation~\eqref{eq: expression of v and g}, we have
\begin{align*}
\widehat{v}(y)=  \sum_{i,j=1}^2\frac{y_{i} y_{j}}{|y|^{4}}\widehat{f}_{i j}(y), \qquad \mbox{ for } y \neq 0.
\end{align*}
Now the existence $v$  is guaranteed by taking inverse Fourier transform of  $ \wh{v}(y)$, however to achieve this we have to first show that $\wh{v}(y)\in L^1(\Rt)$.  Since $f$ is in $\mathcal{S}(\mathbb{R}^2;S^2)$, this implies  $\widehat{v}$ is an integrable function when $\abs{y}\ge 1$. Therefore, we only need to show the integrability of $\widehat{v}$ in a ball around the origin (say in $\{|y| \leq 1\}$). To accomplish this, we must show that $|\wh{v}|$ does not grow faster than $|y|$ in any small neighbourhood of $0$. Consider
\begin{align*}
\widehat{v}(y)&=  \sum_{i,j=1}^2\frac{y_{i} y_{j}}{|y|^{4}}\widehat{f}_{i j}(y)\\
&= \sum_{i,j=1}^2 \frac{y_{i} y_{j}}{|y|^{4}}\left\{\widehat{f}_{ij}(0) + \sum_{k=1}^2 y_k \frac{\partial \widehat{f}_{ij}}{\partial y_k} (0) +\sum_{k,\ell=1}^2 y_k y_\ell \frac{\partial^2 \widehat{f}_{ij}}{\partial y_k \partial y_\ell} (0) + \cdots \right\}, \quad  \mbox{using Taylor's formula}\\
&=  \sum_{i,j, k=1}^2 \frac{y_{i} y_{j} y_k}{|y|^{4}} \frac{\partial \widehat{f}_{ij}}{\partial y_k} (0) +  \sum_{i, j, k,\ell=1}^2 \frac{y_{i} y_{j} y_k y_\ell}{|y|^4} \frac{\partial^2 \widehat{f}_{ij}}{\partial y_k \partial y_\ell} (0) + \cdots +  \text{ lower order terms},  \mbox{ because } \widehat{f}_{ij}(0) = 0.
\end{align*}
From the above relation, we have that
 \begin{align*}
 |\widehat{v}(y)| \leq |y|^{-1}\left(C+ \mathcal{G}(y)\right),\qum{for $|y| \leq 1$}.
 \end{align*}
where $C$ is a constant and $\mathcal{G}(y)$ is a bounded integrable function of $y$. This shows that $\wh{v}(y)\in L^1(\Rt)$, therefore by~\cite[Proposition 3.1]{Dorina}  we have that $v(x)\in C(\Rt)$ and $v(x) \rightarrow 0$ as $\abs{x}\rightarrow \infty$.  To show that $v\in C^{\infty} (\Rt)$, it is enough to show that $y^{\alpha} \wh{v}(y)\in L^1(\Rt)$ for any multi-index $\alpha$. We consider the case when $\abs{\alpha}\ge 1$ because $\abs{\alpha}=0$ case is already considered above. Since $\displaystyle  \widehat{v}(y)=  \sum_{i,j=1}^2 \frac{y_{i} y_{j}}{|y|^{4}}\widehat{f}_{i j}(y)$, this gives $\displaystyle y^{\alpha} \wh{v}(y) =  \sum_{i,j=1}^2 y^{\alpha}\frac{y_{i} y_{j}}{|y|^{4}}\widehat{f}_{i j}(y) \in L^1(\Rt) $ for any multi-index with $\abs{\alpha}\ge 1$. Hence we obtain that $\PD^{\alpha} v\in C(\Rt)$ for all multi-index $\alpha$, and consequently $v\in C^{\infty} (\Rt)$.
\vspace{2mm}\\ Similarly, from~\eqref{eq: expression of v and g} we have that 
\begin{align}
   \wh{g}_{ij}(y) 
= \sigma(i,j)\sum_{k,\ell=1}^2\left( \delta_{ik}\delta_{j\ell} - \frac{y_{i} y_{j}y_{k} y_{\ell}}{|y|^{4}}\right)\wh{f}_{k \ell}(y),
\end{align}
which is an integrable function for each $1\le i,j\le 2$. 
In a similar vain, one can show that $ y^{\alpha}\wh{g}_{ij}(y) \in L^1(\Rt) $ for any multi-index $\alpha$. This further gives $g_{ij}\in C^{\infty}(\Rt) $ and $g_{ij}(x) \rightarrow 0$ as $\abs{x}\rightarrow \infty$ for each $i,j$.

Thus, we obtain the integrability of $\widehat{v}$ and $\wh{g}_{ij}$. Hence, by using the inverse Fourier transform to~\eqref{eq:Fourier Decomposition}, we get the following decomposition:
		\begin{align*}
		f= g+\D^2 v;  \quad \ \delta^{2} g=0.
		\end{align*}
It remains to verify that $g$ and $v$ satisfy the decay properties stated in~\eqref{new_decomp_of_f_no_restriction}.  We use a similar argument used in the proof of~\cite[Theorem 3.4]{Rohit_Suman}. Next using $\displaystyle \wh{v}(y)= \sum_{i,j=1}^2\frac{y_{i} y_{j}}{|y|^{4}}  \widehat{f}_{ij}$ and the fact that $\wh{f}_{ij}(0)=0$ we conclude  
\begin{align*}
 \abs{  D^{\alpha} \wh{v}(y)}\le C \abs{y}^{-1-|\alpha|}, \qum{and} \quad  \abs{i^2_y \wh{v}(y)} =  \abs{y_i y_j\, \wh{v}(y)} \le C \abs{y}.
\end{align*}
This further gives  $\abs{\wh{g}(y)}=\abs{\wh{f}(y)- i^2_y \wh{v}(y)} \le   C \abs{y}$. Moreover, one has 
\begin{align}\label{estimate_of_g}
    \abs{ D^{\alpha} \wh{g}(y)}\le   C \abs{y}^{1-|\alpha|}.
\end{align}
Next we write $g(x)= \int_{\Rt} e^{\I x\cdot y} \wh{g}(y) dy$, and  for any multi-index $\alpha$ obtain
\begin{align}
    x^{\alpha} g(x) = \int_{\Rt}   x^{\alpha} \, e^{\I x \cdot y} \wh{g}(y) dy=  \I^{-|\alpha|}\,\int_{\Rt}    \, \wh{g}(y) \, \p^{\alpha}(e^{\I x \cdot y})  dy.
\end{align}
 At this stage, applying  integration by parts and  using  the estimate of~\eqref{estimate_of_g} we obtain $\abs{x^{\alpha} g(x) }$ is finite if $\abs{\alpha}\le 2$. This gives 
 \begin{align}
     \abs{ g(x) }\le C (1+|x|)^{-2}.
 \end{align}
 This completes the proof of the decay estimate for $g$. By a similar argument, one can establish the corresponding decay estimate for $v$.  This completes the proof of existence part. 
 
 We now focus on the uniqueness part.
 Assume if possible, we have two such decomposition of $f$, that is, there are $g_1,\  g_2, \ v_1$ and $v_2$ satisfying 
\begin{align*}
    g_1 + \D^2 v_1 =  f =     g_2 + \D^2 v_2, \quad \mbox{ and } \quad  \delta^{2}g_1=0 =   \delta^{2}g_2\\
 \Rightarrow   (g_1 -g_2) + \D^2 (v_1 -v_2) =  0, \quad \mbox{ and } \quad  \delta^{2}(g_1-g_2)=0.
\end{align*}
To prove the uniqueness of the decomposition, it is enough to prove that  $ f=0 $ implies  $ g=v=0 $. Now $f=0$ gives $ g+\d^{2}v=0 $ and $ \delta^{2}g=0 $. Since $ g \in \mathcal{S}'(\mathbb{R}^2;S^2) $ and $ v\in \mathcal{S}'(\mathbb{R}^2;\Rb^2) $, where $ \mathcal{S}' $ denotes the space of tempered distributions. Applying Fourier transform  to the equations $ g+\d^{2}v=0 $ and $ \delta^{2}g=0 $, we get $ \widehat{g}(y)+ \I^{2} i_y^2\widehat{v}(y)=0 $ and $ j_y^{2} \widehat{g}(y)=0$. By Lemma~\ref{th: decomposition of f in frequency variable} we have $ \widehat{g}(y)=\widehat{v}(y)=0 $ in $ \Rb^2_0= \Rb^2\setminus\{0\} $, $i.e., $ the support of distributions is contained in $ \{0\}.$ Thus $ \widehat{g} $ and $ \widehat{v} $ can be written as a finite linear combination of derivatives of the Dirac delta distribution, see~\cite[Theorem 3.2.1]{Friedlander}.
Therefore
 \[ \widehat{g} = \sum\limits_{|\alpha|\le p} c_{\alpha} \PD^{\alpha} \delta_{0}, \quad \mbox{ for some positive integer } p, \] 
 where $ \delta_0 $ is the Dirac delta distribution. 
 Again $ \PD^{\alpha}\delta_0 \in \mathcal{S}'(\Rb^2)$ for any multi-index $ \alpha $. Taking the inverse Fourier transform of this equation in the sense of tempered distributions, we obtain that $ g$ is a polynomial of degree at most $ p$. But $ g(x) \rightarrow 0$ as $ |x| \rightarrow \infty$  implies $ g = 0 $ in $ \Rb^2 $.
One can argue similarly and conclude that $ v(x) =0$ in $ \Rb^2 $. \smallskip

 \textbf{Step 2.} We now prove that~\ref{B} implies~\ref{A}.  Suppose $f\in \sch(\mathbb{R}^2;S^2)$ satisfies the decomposition given in equation~\eqref{new_decomp_of_f_no_restriction}, that is, 
$$ f = g +\D^2 v, \qquad \mbox{ with } \quad \delta^2 g=0,$$
where $g$ and $v$ are smooth and satisfy~\eqref{new_decomp_of_f_no_restriction}. This  implies that we can think of $f$, $g$, and $v$ as tempered distributions. Taking the Fourier transform of  $ f = g +\D^2 v $ in the sense of distributions, we obtain
\begin{align*}
    \wh{f}_{ij}(y)= \wh{g}_{ij}(y) - y_{i}y_{j} \wh{v}(y) \implies \wh{v}(y) = \sum_{i,j=1}^2 \frac{y_i y_j \, \wh{f}_{ij}(y)}{|y|^4} \quad \mbox{in the sense of distributions}.
\end{align*}
Next rewrite $\wh{v} $ as 
\begin{align*}
 \wh{v}(y) = \sum_{i,j=1}^2 \frac{y_i y_j \, \wh{f}_{ij}(0)}{|y|^4}+ \sum_{i,j=1}^2 \frac{y_i y_j \,( \wh{f}_{ij}(y) - \wh{f}_{ij}(0))}{|y|^4} =I(y) +II(y).   
\end{align*}
Next using Taylor theorem we can write $ \displaystyle \wh{f}_{ij}(y) - \wh{f}_{ij}(0)=  \sum_{k=1}^2 y_k \frac{\partial \widehat{f}_{ij}}{\partial y_k} (sy)$ for some  real number $s$ with $ 0< s< 1 $.   Since $f$ is Schwartz, and  $ \wh{f}$ is also Schwartz, this implies $$  \wh{f}_{ij}(y) - \wh{f}_{ij}(0)= \sum_{k=1}^2 y_k \frac{\partial \widehat{f}_{ij}}{\partial y_k} (sy) \in \sch(S^2).$$  This further  entails 
$$II(y) = \sum_{i,j=1}^2 \frac{y_i y_j \,( \wh{f}_{ij}(y) - \wh{f}_{ij}(0))}{|y|^4} = \sum_{i,j, k=1}^2\frac{y_i y_jy_k}{\abs{y}^4}  \frac{\partial \widehat{f}_{ij}}{\partial y_k} (sy) \in L^1(\Rt).$$ 
Therefore the inverse  Fourier transform of $II(y)$ is continuous and has a decay at infinity by Riemann-Lebesgue lemma. Similarly, one can show that $ y^{\alpha} II(y)\in L^1(\Rt) $ for any multi-index $\alpha$ with $\abs{\alpha}\ge 1.$  This gives $ \p^{\alpha} \wh{II}(y) $ is continuous, consequently we have that 
\begin{align}\label{smoothness_of_II}
    \wh{II}(y)\in C^{\infty}(\Rt) \qum{ and all the partial derivatives $\p^{\alpha} \wh{II}(y)$ converge to $o$   decay when $|x| \rightarrow\infty $. }
\end{align}
We now focus on the term $\displaystyle I(y) =  \sum_{i,j=1}^2 \frac{y_i y_j \, \wh{f}_{ij}(0)}{|y|^4} $ which is homogeneous of  degree $-2$. This shows that $I(y)$ is not a locally integrable function; however, one could still interpret $I(y)$ as a tempered distribution in the principal value sense (see~\cite[Section 8 in  Chapter 3]{Taylor_book} for details), see also~\cite{Hormander_I}. More details of homogeneous distributions can be found in the Appendix~\ref{appen_homo_dis}.
%\ssc{Put details of homogeneous distribution of degree $-2$ here or in the appendix $?$}
% \antti{Needs a citation, Fourier analysis book, or distribution theory book.}). 
Assume that $I(y)$ is viewed as a tempered distribution, our next goal is to find the inverse Fourier transform of $I$. This will give rise a $\log$  term. To see this, we compute $ \PD_{ij} (\log \abs{y}) $. This is given by $\PD_{ij} (\log \abs{y}) = \frac{\delta_{ij}}{\abs{y}^2}- \frac{2 y_iy_j}{\abs{y}^4}.$
% \begin{align*}
%  &\PD_{ij} (\log \abs{y}) = \frac{\delta_{ij}}{\abs{y}^2}- \frac{2 y_iy_j}{\abs{y}^4}
%  \end{align*}
 Taking Fourier transform of this, we obtain 
 \begin{align*}
 &\widehat{\frac{2 y_iy_j}{\abs{y}^4}}(x) = \widehat{\delta_{ij} |y|^{-2}} + x_i x_j \widehat{\log\abs{y}}\implies \widehat{\frac{2 y_iy_j}{\abs{y}^4}}(x) = \delta_{ij} \log\abs{x} +c  \frac{x_i x_j}{\abs{x}^2}.
\end{align*}
To compute the Fourier transform of $\abs{y}^{-2}$ in $\sch'(\Rt)$ we have used the fact that $ \log \abs{x}$ is the fundamental solution of $ \Delta$ in $\mathbb{R}^2$, $i.e., \Delta (\log \abs{x}) = c\, \delta$. This further entails $ \abs{y}^2 \wh{\log|x|}(y)= C.  $
\vspace{2mm}\\
This implies after taking inverse Fourier transform and using~\eqref{smoothness_of_II} we write $v$ as
%\antti{Maybe quantify what "nice term" means.} 
\begin{align}
    v= \wh{I}+ \wh{II}=  \sum_{i,j=1}^2\wh{f}_{ij}(0) \left( \delta_{ij} \log\abs{x} +c  \frac{x_i x_j}{\abs{x}^2}\right) + \mbox{$C^{\infty}$ smooth  decaying term}.
\end{align}
Since we started with the decomposition result~\eqref{new_decomp_of_f_no_restriction}, this implies for  $v \in C^{\infty}(\Rt) $ and to have a decay at infinity, we must have 
\[  \sum_{i,j=1}^2 \wh{f}_{ij}(0) \left( \delta_{ij} \log\abs{x} +c  \frac{x_i x_j}{\abs{x}^2}\right)=0 \qum{for all $x\in \Rt$}.\]
This implies  as $\abs{x}\rightarrow \infty$ we have that the $\log(\abs{x})$ term will blow up and the term $ \frac{x_i x_j}{\abs{x}^2}$ is direction dependent fixed constant which is not zero. This forces $  \sum_{i,j=1}^2\omega_i  \wh{f}_{ij}(0)  \omega _j=0$ for all unit vectors $\omega$.  As a result we have that $\wh{f}_{ij}(0)=0$ for all $i,j$. This also implies $\sum_{i=1}^2 \wh{f}_{ii}(0)  \log (\abs{x})=0$. Therefore, we must have $ \wh{f}_{ij} (0)=\int_{\Rt} f_{ij} =0$. The proof is complete.
\end{proof}
   % \ssc{We will discuss the remark later.}
% \begin{remark}
% In general, the following decomposition result holds in $\Rb^2$
% for any $f\in \sch(S^2)$ \[  f = g+ \D^2 v \qum{for some $v=v_1+v_2+v_3$ and $g\in C^{\infty}(\Rt) $ with $\delta^2 g=0$,}\] where $v_1\in C^{\infty}(\Rt), v_2,v_3\in C^{\infty}(\Rt\setminus\{0\}) $ and $v_2$ solves $ \Delta v_2= \delta$ and $v_3(x) = \displaystyle \frac{x_ix_j}{\abs{x}^2}$ is a homogeneous polynomial of degree $0$ in $\Rt\setminus\{0\}$. Note that $v_1$ and $g$ have decay at infinity.
% \end{remark}
\subsection{Decomposition result for elastic tensors}
\label{sec:elastic-decomp}

In this section, we prove that an elastic tensor field in $2$ dimensions and be uniquely decomposed into solenoidal and potential parts. The terms solenoidal and potential tensor fields have a different, but analogous, meaning than usually in tensor field decomposition literature, or earlier in this article. The operators $H$ and $K$ and their formal adjoints were defined in Section~\ref{sec:Def and notation}.

\begin{definition}
An elastic tensor field $f \in C^{\infty}(\mathbb{R}^n;E^2(n))$ is called
\begin{enumerate}
    \item solenoidal if $H^*f = K^*f = 0$, and
    \item potential if $f = Hv + Ku$ for some $u \in C^\infty(\mathbb{R}^n;\mathbb{R}^n)$ and $v \in C^\infty(\mathbb{R}^n;E^1(n))$.
\end{enumerate}
\end{definition}

% In Section~\ref{sec:Def and notation}, we defined the operator $H \colon \mathcal{S}(\mathbb{R}^n;E^1(n)) \to \mathcal{S}(\mathbb{R}^n;E^2(n))$ and its formal $L^2$-adjoint $H^* \colon \mathcal{S}(\mathbb{R}^n;E^2(n)) \to \mathcal{S}(\mathbb{R}^n;E^1(n))$ by
% \begin{equation}
% (Hv)_{ijkl}
% =
% \frac{1}{2}
% \left(
% \frac{\partial^2 v_{kl}}{\partial x_i\partial x_j}
% +
% \frac{\partial^2 v_{ij}}{\partial x_k\partial x_l}
% \right)
% \quad\text{and}\quad
% (H^*w)_{ij}
% =
% \sum_{k,l}
% \frac{\partial^2w_{ijkl}}{\partial x_k\partial x_l}.
% \end{equation}

We will heavily rely on Fourier transforms to prove existence of solenoidal-potential decompositions of elastic tensor fields. For this purpose, we define the following operators, which can be understood as pointwise Fourier analogs of the operators $H$ and $K$ defined in Section~\ref{sec:Def and notation}. Given $y \in \mathbb{R}^n \setminus \{0\}$, we define the operators $\widehat{H}_y \colon E^1(n) \to E^2(n)$ and $\widehat{H}^*_y \colon E^2(n) \to E^1(n)$ by
\begin{equation}
(\widehat{H}_y\widehat v)_{ijkl}
=
\frac{1}{2}(y_iy_j\widehat v_{kl} + y_ky_l\widehat v_{ij})
\quad\text{and}\quad
(\widehat{H}^*_y\widehat w)_{ij}
=
\sum_{k,l}
y_ky_l\widehat w_{ijkl}.
\end{equation}
Here $\widehat v \in E^1(n)$ and $\widehat w \in E^2(n)$ are only vectors and tensors. However, as the notation suggests taking Fourier transforms gives $\widehat{(Hv)}(y) = \widehat{H}_y\widehat v(y)$ where $\widehat v$ is now the Fourier transform of $v \in \mathcal{S}(\mathbb{R}^n;E^1(n))$.

% Similarly, we define the operator $K \colon \mathcal{S}(\mathbb{R}^n;\Rb^n) \to \mathcal{S}(\mathbb{R}^n;E^2(n))$ by $Ku =\varepsilon(du \otimes I)$ where $I$ is the identity matrix. In coordinates, we have
% \begin{equation}
% (Ku)_{ijkl}
% =
% \frac14
% \left(
% \frac{\partial u_j}{\partial x_i}
% +
% \frac{\partial u_i}{\partial x_j}
% \right)
% \delta_{kl}
% +
% \frac14
% \left(
% \frac{\partial u_k}{\partial x_l}
% +
% \frac{\partial u_l}{\partial x_k}
% \right)
% \delta_{ij}.
% \end{equation}
% The formal $L^2$-adjoint of $K$ is the operator $K^* \colon \mathcal{S}(\mathbb{R}^n;E^2(n)) \to \mathcal{S}(\mathbb{R}^n;\Rb^n)$ defined by
% \begin{equation}
% (K^*w)_i
% =
% -
% \sum_{j,k}
% \frac{\partial w_{ijkk}}{\partial x_j}.
% \end{equation}
Similarly, the pointwise Fourier counterparts to the operators $K$ and $K^*$ (also defined in Section~\ref{sec:Def and notation}) are
\begin{equation}
(\widehat{K}_y\widehat{u})_{ijkl}
=
\frac{1}{4}
(y_i\widehat u_j + y_j\widehat u_i)
\delta_{kl}
+
\frac{1}{4}
(y_k\widehat u_l + y_l\widehat u_k)
\delta_{ij},
\quad\text{and}\quad
(\widehat{K}^*_y\widehat w)_i
=
\sum_{j,k}
y_jw_{ijkk}.
\end{equation}
We aim to use these differential operators to characterize the kernel of $X^2$. The approach is similar to the one taken in Section~\ref{sec:Decomposition Results}. First, we show that certain pointwise decompositions exist on the Fourier domain. Then to obtain the decomposition for tensor fields we must analyze the point dependence of the Fourier decomposition and to use Fourier inversion.

The following pointwise decomposition in the Fourier domain was proved in~\cite{ilmavirta2025elasticraytransform}. 

\begin{lemma}[{\cite[Lemma 10]{ilmavirta2025elasticraytransform}}]
Let $y \in \mathbb{R}^n \setminus \{0\}$. Then the space of elastic $2$-tensors admits the following (interior) direct sum decomposition
\[ E^2(n) = A_y \oplus B_y \oplus C_y, \]
where
\begin{align*}
    A_y &:= \{ \widehat{K}_y \widehat u : \widehat u \in \mathbb{R}^n,\; u \perp y\}, 
    \qquad B_y := \{ \widehat{H}_y \widehat v : \widehat v \in E^1(n) \} \\
     C_y &:= \{ \widehat g \in E^2(n) : \widehat{K}_y^{*} \widehat g = \widehat{H}_y^{*} \widehat g = 0 \}.
\end{align*}
% \[ A_p := \{ K_p W : W \in \mathbb{R}^n,\; W \perp p \}, \]
% \[ B_p := \{ H_p h : h \in E^1(n) \}, \]
% and
% \[ C_p := \{ S \in E^2(n) : K_p^{*} S = H_p^{*} S = 0 \}. \]
Moreover, the projections onto the subspaces $A_y$, $B_y$, and $C_y$ depend $0$-homogeneously on $y$.
\end{lemma}

We add on to the above lemma by showing how to compute $\widehat u$, $\widehat v$ and $\widehat g$ in terms of $\widehat f$ (see Lemma~\ref{lem:E2-decomposition}). The explicit formulas, in particular, the homogeneities in $y$ are important to make conclusions about Fourier transforms of certain tensor fields later.

\begin{lemma}[Decomposition of elastic $2$-tensors]
\label{lem:E2-decomposition}
Let $y \in \mathbb{R}^n\setminus \{0\}$. For any elastic tensor $\widehat f \in E^2(n)$ there are unique $\widehat u \in \mathbb{R}^n$, $\widehat v \in E^1(n)$ and $\widehat g \in E^2(n)$ so that $\widehat f = \widehat{H}_y \widehat v + \widehat{K}_y \widehat u + \widehat g$ with $\widehat{H}^*_y\widehat g = \widehat{K}^*_y\widehat g = 0$ and $\widehat u \perp y$.

Moreover, $\widehat v$ and $\widehat u$ can be explicitly computed in terms of $\widehat f$ by
\begin{equation}
\label{eqn:hat-u}
\widehat u_i
=
\frac{4}{n-1}
\left(
\frac{1}{\abs{y}^2}
\sum_j
\varepsilon_{ij}(y)
(\widehat{K}^*_y\widehat f)_j
-
\frac{1}{\abs{y}^4}
\sum_k
y_k(\widehat{H}^*_y\widehat f)_{ik}
+
\frac{y_i}{\abs{y}^6}
\sum_{i,k}
y_iy_k
(\widehat{H}^*_y\widehat f)_{ik}
\right)
\end{equation}
and
\begin{equation}
\label{eqn:hat-v}
\widehat v_{ij}
=
\frac{2}{\abs{y}^4}
(\widehat{H}^*_y\widehat f)_{ij}
-
\frac{y_iy_j}{\abs{y}^8}
\sum_{k,l}y_ky_l
(\widehat{H}^*_y\widehat f)_{kl}
-
\frac{1}{2\abs{y}^2}(y_j\widehat u_i + y_i\widehat u_j)
\end{equation}
where $\varepsilon_{ij}(y) = \delta_{ij}-\frac{y_iy_j}{\abs{y}^{2}}$.
\end{lemma}

\begin{proof}
We start by proving uniqueness of such a decomposition. Assume that $\widehat{f} \in E^2(n)$, $ \widehat g \in E^2(n)$, $\widehat v \in E^1(n)$ and $\widehat u \in \mathbb{R}^n$ are such that
\begin{equation}
\label{eqn:assumed-decomposition}
\widehat f
=
\widehat{H}_y\widehat v
+
\widehat{K}_y\widehat u
+
\widehat g
\quad\text{with}\quad
\widehat{H}^*_y\widehat g
=
\widehat{K}^*_y\widehat g
=
0
\quad\text{and}\quad
\widehat u \perp y.
\end{equation}
Uniqueness of the decomposition is proved by showing that given this decomposition of $\widehat f$ then $\widehat u$ and $\widehat v$ are necessarily of the form~\eqref{eqn:hat-u} and~\eqref{eqn:hat-v}.

To show that~\eqref{eqn:hat-u} and~\eqref{eqn:hat-v} we start by applying the operator $\widehat{K}^*_y$ to~\eqref{eqn:assumed-decomposition} to obtain
\begin{equation}
\label{eqn:k-star-applied-to-decomp}
\widehat{K}^*_y\widehat f
=
\widehat{K}^*_y\widehat{H}_y\widehat v
+
\widehat{K}^*_y\widehat{K}_y\widehat u
\end{equation}
since $\widehat{H}^*_y\widehat g = \widehat{K}^*_y\widehat g = 0$. We compute that in components
\begin{equation}
\label{eqn:k-star-f-decomp-part-1}
\begin{split}
(\widehat{K}^*_y\widehat{H}_y\widehat v)_i
&=
\sum_{k,j}
y_j(\widehat{H}_y\widehat v)_{ijkk}=
\frac{1}{2}
\sum_{k,j}
\left(
y_iy_j^2\widehat v_{kk}
+
y_jy_k^2\widehat v_{ij}
\right)
\\
&=
\frac{\abs{y}^2}{2}
\left(
\sum_{k}
y_i\widehat v_{kk}
+
\sum_{j}
y_j\widehat v_{ij}
\right)=
\frac{\abs{y}^2}{2}
\sum_{k}
\left(
y_i\widehat v_{kk}
+
y_k\widehat v_{ik}
\right).
\end{split}
\end{equation}
For the second part of the right-hand side of~\eqref{eqn:k-star-applied-to-decomp}, we get
\begin{equation}
\begin{split}
(\widehat{K}^*_y\widehat{H}_y\widehat u)_i
&=
\sum_{k,j}
y_j(\widehat{K}_y\widehat u)_{ijkk}=
\frac{1}{4}
\sum_{k,j}
\left(
y_j(y_i\widehat u_j + y_j\widehat u_i)\delta_{kk}
+
y_j(y_k\widehat u_k+y_k\widehat u_k)\delta_{ij}
\right).
\end{split}
\end{equation}
Using the fact that $\widehat u \perp y$, the above reduces to
\begin{equation}
\label{eqn:k-star-f-decomp-part-2}
\begin{split}
(\widehat{K}^*_y\widehat{H}_y\widehat u)_i
&=
\frac14
\sum_{k,j}
y^2_j\widehat u_i\delta_{kk}=
\frac{n}{4}
\abs{y}^2
\widehat u_i.
\end{split}
\end{equation}
Combining~\eqref{eqn:k-star-f-decomp-part-1} and~\eqref{eqn:k-star-f-decomp-part-2} gives
\begin{equation}
\label{eqn:hat-k-f}
(\widehat{K}^*_y\widehat f)_i
=
\frac{n}{4}
\abs{y}^2
\widehat u_i
+
\frac{\abs{y}^2}{2}
\sum_{k}
\left(
y_i\widehat v_{kk}
+
y_k\widehat v_{ik}
\right).
\end{equation}
We will make use of this formula later, but first we will derive a formula for $\widehat{H}^*_y\widehat f$. We compute in coordinates that
\begin{equation}
\label{eqn:h-star-f-decomp-part-1}
\begin{split}
(\widehat{H}^*_y\widehat{H}_y\widehat v)_{ij}
&=
\sum_{k,l}
y_ky_l(\widehat{H}_y\widehat v)_{ijkl}
=
\frac12
\sum_{k,l}
y_ky_l
\left(
y_iy_j\widehat v_{kl}
+
y_ky_l\widehat v_{ij}
\right)
=
\frac12
\left(
y_iy_j
\sum_{k,l}
y_ky_l\widehat v_{kl}
+
\abs{y}^4
\widehat v_{ij}
\right)
.
\end{split}
\end{equation}
Similarly, we compute that
\begin{equation}
\begin{split}
(\widehat{H}^*_y\widehat{K}_y\widehat u)_{ij}
&=
\sum_{k,l}
y_ky_l
(\widehat{K}_y\widehat u)_{ijkl}
=
\frac{1}{4}
\sum_{k,l}
y_ky_l
\left(
(y_i\widehat u_j + y_j\widehat u_i)\delta_{kl}
+
(y_k\widehat u_l + y_l\widehat u_k)\delta_{ij}
\right)
\\
&=
\frac{1}{4}
\left(
\abs{y}^2
(y_j\widehat u_i + y_i\widehat u_i)
+
2
\abs{y}^2
(\widehat u \cdot y)
\delta_{ij}
\right).
\end{split}
\end{equation}
Applying the fact that $\widehat u \perp y$, reduces the above to
\begin{equation}
\label{eqn:h-star-f-decomp-part-2}
\begin{split}
(\widehat{H}^*_y\widehat{K}_y\widehat u)_{ij}
=
\frac{\abs{y}^2}{4}
\left(
y_j\widehat u_i+y_i\widehat u_j
\right).
\end{split}
\end{equation}
Since $\widehat{H}^*_y\widehat g = \widehat{K}^*_y\widehat g = 0$,  combining~\eqref{eqn:h-star-f-decomp-part-1} and~\eqref{eqn:h-star-f-decomp-part-2} gives
\begin{equation}
\label{eqn:h-star-f}
(\widehat{H}^*_y\widehat f)_{ij}
=
\frac{1}{2}
\left(
y_iy_j\sum_{k,l}
y_ky_l\widehat v_{kl}
+
\abs{y}^4
\widehat v_{ij}
\right)
+
\frac{\abs{y}^2}{4}
\left(
y_j\widehat u_i+y_i\widehat u_j
\right).
\end{equation}
We contract this equation with $y$ twice (once in both the $i$ and $j$ indices) and use the definition of $\widehat{H}^*_y$ to obtain
\begin{equation}
\label{eqn:f-contracted-al-indices}
\begin{split}
\sum_{i,j,k,l}
y_iy_jy_ky_l
\widehat f_{ijkl}
&=
\sum_{i,j}
y_iy_j(\widehat{H}^*_y\widehat f)_{ij}
\\
&=
\frac12
\left(
\sum_i
y_i^2
\right)
\left(
\sum_i
y_i^2
\right)
\left(
\sum_{k,l}
y_ky_l\widehat v_{kl}
\right)
+
\frac12\abs{y}^4
\sum_{i,j}
y_iy_j\widehat v_{ij}
\\
&\quad
+
\frac{\abs{y}^2}{2}
\left(
\left(
\sum_j y_j^2
\right)
\left(
\sum_i
y_i\widehat u_i
\right)
+
\left(
\sum_i y_i^2
\right)
\left(
\sum_k
y_j\widehat u_j
\right)
\right)
\\
&=
\abs{y}^4
\sum_{k,l}y_ky_l\widehat v_{kl}
+
\abs{y}^4(\widehat u \cdot y)=
\abs{y}^4
\sum_{k,l}y_ky_l\widehat v_{kl}.
\end{split}
\end{equation}
We now go back to~\eqref{eqn:k-star-f-decomp-part-1} and~\eqref{eqn:k-star-f-decomp-part-2} and use them to compute
\begin{equation}
\begin{split}
\sum_{i}
y_i(\widehat{K}_y^*\widehat f)_i
&=
\sum_{i}
y_i(\widehat{K}_y^*\widehat{H}_y\widehat v)_i
+
\sum_{i}
y_i(\widehat{K}_y^*\widehat{K}_y\widehat u)_i
\\
&=
\frac{\abs{y}^2}{2}
\sum_{i}
y_i
\left(
y_i
\sum_{k}
\widehat v_{kk}
+
\sum_{k}
y_k\widehat v_{ik}
\right)
+
\sum_{i}
y_i
\left(
\frac{n}{2}
\abs{y}^2\widehat u_i
\right)
\\
&=
\frac{\abs{y}^4}{2}
\sum_{k}
\widehat v_{kk}
+
\frac{\abs{y}^2}{2}
\sum_{i,k}
y_iy_k
\widehat v_{ik}
+
\frac{n}{2}\abs{y}^2(\widehat u \cdot y)
\\
&=
\frac{\abs{y}^4}{2}
\sum_{k}
\widehat v_{kk}
+
\frac{\abs{y}^2}{2}
\sum_{i,k}
y_iy_k
\widehat v_{ik}.
\end{split}
\end{equation}
From the above equation, we compute using~\eqref{eqn:f-contracted-al-indices} that
\begin{equation}
\begin{split}
\sum_k
\widehat v_{kk}
&=
\frac{2}{\abs{y}^4}
\sum_i
y_i(\widehat{K}^*_y\widehat f)_i
-
\frac{1}{2\abs{y}^2}
\sum_{i,k}
y_iy_k\widehat v_{ik}
=
\frac{1}{\abs{y}^4}
\left(
2
\sum_i
y_i(\widehat{K}^*_y\widehat f)_i
-
\frac{1}{\abs{y}^2}
\sum_{i,j,k,l}
y_iy_jy_ky_l
\widehat f_{ijkl}
\right).
\end{split}
\end{equation}
We insert this into~\eqref{eqn:hat-k-f} to get
\begin{equation}
\label{eqn:auxiliary-equation-1}
\begin{split}
\frac{n}{4}\abs{y}^2\widehat u_i
+
\frac{\abs{y}^2}{2}
\sum_k y_k\widehat v_{ik}
&=
(\widehat{K}^*_y\widehat f)_i
-
\frac{\abs{y}^2}{2}
y_i
\sum_k \widehat v_{kk}
\\
&=
(\widehat{K}^*_y\widehat f)_i
-
\frac{y_i}{2\abs{y}^2}
2\sum_{j}y_j(\widehat{K}^*_y\widehat f)_j
-
\frac{y_i}{2\abs{y}^4}
\sum_{i,j,k,l}
y_iy_jy_ky_l
\widehat f_{ijkl}
.
\end{split}
\end{equation}
Then we go back to~\eqref{eqn:h-star-f} to compute
\begin{equation}
\begin{split}
\sum_{j,k,l}y_jy_ky_l\widehat f_{ijkl}
&=
\sum_{j}
y_j(\widehat{H}^*_y\widehat f)_{ij}
\\
&=
\sum_jy_j
\left(
\frac12
\left(
y_iy_j
\sum_{k,l}
y_ky_l\widehat v_{kl}
+
\abs{y}^4
\widehat v_{ij}
\right)
+
\frac{\abs{y}^4}{4}
\left(
y_j\widehat u_i
+
y_i\widehat u_j
\right)
\right)
\\
&=
\frac{\abs{y}^2}{2}
y_i
\sum_{k,l}
y_ky_l
\widehat v_{kl}
+
\frac{\abs{y}^4}{2}
\sum_jy_j\widehat v_{ij}
+
\frac{\abs{y}^4}{4}\widehat u_i
+
\frac{\abs{y}^4}{4}y_i(\widehat u \cdot y)
\\
&=
\frac{\abs{y}^2}{2}
y_i
\sum_{k,l}
y_ky_l
\widehat v_{kl}
+
\frac{\abs{y}^4}{2}
\sum_jy_j\widehat v_{ij}
+
\frac{\abs{y}^4}{4}\widehat u_i.
\end{split}
\end{equation}
Rearranging the above equation and inserting~\eqref{eqn:f-contracted-al-indices} gives
\begin{equation}
\label{eqn:auxiliary-equation-2}
\frac{\abs{y}^4}{4}\widehat u_i
+
\frac{\abs{y}^4}{2}
\sum_jy_j\widehat v_{ij}
=
\sum_{j,k,l}
y_jy_ky_l
\widehat f_{ijkl}
-
\frac{y_i}{2\abs{y}^2}
\sum_{i,j,k,l}
y_iy_jy_ky_l
\widehat f_{ijkl}.
\end{equation}
Then multiplying~\eqref{eqn:auxiliary-equation-1} by $\abs{y}^2$ and then subtracting~\eqref{eqn:auxiliary-equation-2} gives
\begin{equation}
\begin{split}
\frac{n-1}{4}\abs{y}^4\widehat u_i
&=
\abs{y}^2
\left(
(\widehat{K}^*_y\widehat f)_i
-
\frac{y_i}{\abs{y}^2}\sum_m y_m(\widehat{K}^*_y\widehat f)_m
+
\frac{y_i}{2\abs{y}^4}
\sum_{i,j,k,l}
y_iy_jy_ky_l\widehat f_{ijkl}
\right)
\\
&\quad
-
\sum_{j,k,l}
y_jy_ky_l\widehat f_{ijkl}
+
\frac{y_i}{2\abs{y}^2}
\sum_{i,j,k,l}
y_iy_jy_ky_l\widehat f_{ijkl}
\\
&=
\abs{y}^2
\sum_j
\varepsilon_{ij}(y)
(\widehat{K}^*_y\widehat f)_j
-
\sum_k
y_k(\widehat{H}^*_y\widehat f)_{ik}
+
\frac{y_i}{\abs{y}^2}
\sum_{i,k}
y_iy_k
(\widehat{H}^*_y\widehat f)_{ik},
\end{split}
\end{equation}
where $ \varepsilon_{ij}(y)
=
\delta_{ij}
-\frac{y_iy_j}{\abs{y}^2}.$
% \begin{equation}
% \varepsilon_{ij}(y)
% \coloneqq
% \delta_{ij}
% -
% \frac{y_iy_j}{\abs{y}^2}.
% \end{equation}
This proves formula~\eqref{eqn:hat-u} since the exact form can be obtained by diving by a power of $\abs{y}$ appropriately. Next, we will find a formula for $\widehat v$ in terms of $\widehat f$ and $\widehat u$ (and thus in terms of $\widehat f$ alone by~\eqref{eqn:hat-u}). In order to find the formula for $\widehat v$ rewrite~\eqref{eqn:h-star-f} and use~\eqref{eqn:f-contracted-al-indices} to compute
\begin{equation}
\begin{split}
\widehat v_{ij}
&=
\frac{1}{\abs{y}^4}
\left(
2(\widehat{H}^*\widehat f)_{ij}
-
y_iy_j
\sum_{k,l}
y_ky_l
\widehat v_{kl}
-
\frac{\abs{y}^2}{2}(y_j\widehat u_i+y_i\widehat u_j)
\right)
\\&=
\frac{1}{\abs{y}^4}
\left(
2(\widehat{H}^*\widehat f)_{ij}
-
\frac{y_iy_j}{\abs{y}^4}
\sum_{k,l}
(\widehat{H}^*_y\widehat f)_{kl}
-
\frac{\abs{y}^2}{2}(y_j\widehat u_i+y_i\widehat u_j)
\right).
\end{split}
\end{equation}
This proves the second formula (Equation~\eqref{eqn:hat-v}) of the lemma. The exact form is obtained by distributing the powers of $\abs{y}$ above.

Finally, to find $\widehat g$ in terms of $\widehat f$ we can simply substract $\widehat{H}_y\widehat v + \widehat{K}_y\widehat u$ from $\widehat f$ since $\widehat u$ and $\widehat v$ can be written in terms of $\widehat f$. Given the desired decomposition, we have derived formulas for $\widehat v$, $\widehat u$ and $\widehat g$ in terms of $\widehat f$. This shows that such a decomposition is necessarily unique. To prove existence, we can simply define $\widehat v$, $\widehat u$ and $\widehat g$ using the formulas above. This completes the proof of our lemma.
\end{proof}
With the pointwise Fourier decomposition from Lemma~\ref{lem:E2-decomposition} at hand, we are able to prove a solenoidal decomposition of tensor fields in $\mathcal{S}(\mathbb{R}^2;E^2(2))$. A similar decompositions were shown to hold in $\mathcal{S}(\mathbb{R}^n;E^2(n))$ for $n \geq 3$ in~\cite{ilmavirta2025elasticraytransform}. The explicit form obtained for the parts of pointwise decomposition allow us to extend their result into $2$ dimensions.

\begin{theorem}
\label{th:decomposition_elastic}
Let $ f \in \mathcal{S}(\mathbb{R}^2;E^2(2))$ be an elastic $2$-tensor field satisfying $\int_{\Rb^2} f(x)\,\d x = 0$. Then there exist an elastic $2$-tensor field $g \in C^\infty(\mathbb{R}^2;E^2(2))$, a vector field $u \in C^\infty(\mathbb{R}^2;\mathbb{R}^2)$, and an elastic $1$-tensor field $v \in C^\infty(\mathbb{R}^2;E^1(2))$ such that   \begin{align}
\label{new_decomp_of_f_no_restriction_elastic}
f(x) = H v(x) + K u(x) + g(x)
\quad\text{with}\quad
K^{*} g(x) = H^{*} g(x) = 0 
\quad\text{for}\quad
x \in \Rb^2.
\end{align}	
Moreover, the $g$ and $Hv + Ku$ are uniquely determined by $f$, and $v$ and $u$ satisfy the decay estimates
\begin{equation}
\label{eqn:decay-elastic}
\abs{u_i(x)}
\leq
C(1+\abs{x})^{-1}
\quad\text{and}\quad
\abs{\partial_kv_{ij}(x)}
\leq
C(1+\abs{x})^{-1}.
\end{equation}
\end{theorem}

\begin{proof}
For any $y \in \mathbb{R}^2 \setminus \{0\}$ the Fourier transform of $f$ can be decomposed as
\begin{equation}
\widehat f(y)
=
\widehat{H}_y\widehat v(y)
+
\widehat{K}_y\widehat u(y)
+
\widehat g(y),
\quad\text{where}\quad
\widehat{H}_y^*\widehat g(y)
=
\widehat{K}_y^*\widehat g(y)
=
0
\quad\text{and}\quad
\widehat u(y) \perp y
\end{equation}
according to Lemma~\ref{lem:E2-decomposition}. The components $\widehat v_{ij}(y)$ and $\widehat u_i(y)$ are given in equations~\eqref{eqn:hat-u} and~\eqref{eqn:hat-v}. From these equations we observe that $\widehat v_{ij}(y)$ is homogeneous of degree $-2$ in $y$ and $\widehat u_i(y)$ is homogeneous of degree $-1$ in $y$ and the components are smooth outside the origin in $\mathbb{R}^2$. Moreover, the components satisfy the estimates
\begin{equation}
\label{eqn:integrability-estimate}
\abs{\widehat v_{ij}(y)}
\leq
C\abs{y}^{-2}
\quad\text{and}\quad
\abs{\widehat u_{i}(y)}
\leq
C\abs{y}^{-1}
\quad\text{for}\quad
\abs{y} \leq 1.
\end{equation}
Additionally, equations~\eqref{eqn:hat-u} and~\eqref{eqn:hat-v} imply that $\widehat v_{ij}(y)$ and $\widehat u_i(y)$ are rapidly decaying in $y$ since $\widehat f_{ijkl}(y)$ is Schwartz. For $\widehat u_i(y)$ estimate~\eqref{eqn:integrability-estimate} and rapid decay imply that $u$ is integrable in $\mathbb{R}^2$. To see that $\widehat v_{ij}(y)$ is also integrable we need to use the condition that $\int_{\mathbb{R}^2} f(x)\,\d x = \widehat f(0) = 0$ to get an improved estimate. We use Taylor series expansion at $y = 0$ to compute that
\begin{equation}
(\widehat{H}^*_y\widehat f)_{ij}(y)
=
\sum_{k,l}
y_ky_l\widehat f_{ijkl}(0)
+
\sum_{k,l}
y_ky_l\widehat R_{ijkl}(y), \quad \mbox{and}\quad(\widehat{K}^*_y\widehat f)_{i}(y)
=
\sum_{j,k}
y_j\widehat f_{ijkk}(0)
+
\sum_{j,k}
y_j\widehat Q_{ijkk}(y),
\end{equation}
% and
% \begin{equation}
% (\widehat{K}^*_y\widehat f)_{i}(y)
% =
% \sum_{j,k}
% y_j\widehat f_{ijkk}(0)
% +
% \sum_{j,k}
% y_j\widehat Q_{ijkk}(y)
% \end{equation}
where the error terms $\widehat R_{ijkl}(y)$ and $\widehat Q_{ijkl}(y)$ are of the order $\mathcal{O}(y)$ as $y \to 0$. Inserting the above together with $\widehat f_{ijkl}(0)$ into equation~\eqref{eqn:hat-v} and tracing throught the notation we find that
\begin{equation}
\abs{\widehat v_{ij}(y)}
\leq
C\abs{y}^{-1}
\quad\text{for}\quad
\abs{y} \leq 1.
\end{equation}
Similar to $\widehat u_i$, this estimate implies that $\widehat v_{ij}(y)$ is integrable in $\mathbb{R}^2$.

We have shown that $\widehat v_{ij}(y)$ and $\widehat u_i(y)$ are integrable in $\mathbb{R}^2$ and rapidly decaying. It directly follows that the inverse Fourier transforms of these functions are smooth in $\mathbb{R}^2$ up to the origin. This is a standard argument in Fourier analysis using differentiation under the integral sign. Thus we have found a smooth vector field $u \in C^\infty(\mathbb{R}^2;\mathbb{R}^2)$ and a smooth elastic $1$-tensor field $v \in C^\infty(\mathbb{R}^2;E^1(2))$ required for the decomposition. The decay estimates for $u$ and $v$ in equation~\eqref{eqn:decay-elastic} can be shown by integration by parts using homogeneity of the respective Fourier transforms (cf. Proof of theorem~\ref{th:new_decomposition})

Once we have found the required $u$ and $v$ for the decomposition, the same line of reasoning can be used to prove the smoothness of
\begin{equation}
g(x) = f(x) - Hu(x) - Ku(x).
\end{equation}
Uniqueness of the decomposition follows from uniqueness of the decomposition in Lemma~\ref{lem:E2-decomposition}.
\end{proof}

% \begin{theorem}
% \label{th:decomposition_elastic}
% Let $ f \in \mathcal{S}(E^2(2))$ be an elastic $2$-tensor field satisfying $\int_{\Rb^2} f(x)\,\d x = 0$. Then there exist an elastic $2$-tensor field $g \in C^\infty(\mathbb{R}^2;E^2(2))$, a vector field $u \in C^\infty(\mathbb{R}^2;\mathbb{R}^2)$, and an elastic $1$-tensor field $v \in C^\infty(\mathbb{R}^2;E^1(2))$ such that   \begin{align}
% \label{new_decomp_of_f_no_restriction_elastic}
% f(x) = H v(x) + K u(x) + g(x)
% \quad\text{with}\quad
% K^{*} g(x) = H^{*} g(x) = 0 
% \quad\text{for}\quad
% x \in \Rb^2.
% \end{align}	
% Moreover, the $g$ and $Hv + Ku$ are uniquely determined by $f$, and $v$ and $u$ satisfy the decay estimates
% \begin{equation}
% \label{eqn:decay-elastic}
% \abs{u_i(x)}
% \leq
% C(1+\abs{x})^{-1}
% \quad\text{and}\quad
% \abs{\partial_kv_{ij}(x)}
% \leq
% C(1+\abs{x})^{-1}.
% \end{equation}
% \end{theorem}

\section{Applications to  momentum ray transforms}\label{sec:Application to moments} 
% \ssc{Maybe we could try to write it down for symmetric $m$ tensors lifting the assumption $1\le k < \min{m,n-1}$.}

In this section, our aim is to prove an injectivity result for momentum ray transforms. The result we prove here improves the earlier result~\cite[Theorem 4.3]{Rohit_Suman}, which had a dimensional restriction. Because of the restriction, the result~\cite[Theorem 4.3]{Rohit_Suman} was not true for $m=2$ and $n=2$. 

\begin{definition}\label{def:integral moments}
For a non-negative integer $q \geq 0$, the $q$-th integral moment transform of a symmetric $2$-tensor field is the  operator $I^q :{\Sc}(\mathbb{R}^2;S^2)\rightarrow{\Sc}(T\Sb^1)$ given by~\cite{Sharafutdinov_1986_momentum}:
\begin{align}\label{eq:def momentum ray}
 I^q f(x,\xi)=  \sum_{i,j=1}^2\int\limits_{-\infty}^\infty t^q f_{ij}(x+t\xi)\,\xi_{i}\xi_{j} dt.
\end{align}
\end{definition}
\noindent We also define the operator $J^q: \Sc(\mathbb{R}^2;S^2) \longrightarrow C^\infty(\Rb^2 \times (\Rb^{2} \setminus \{0\}))$ by extending $I^q$ to  $ \Rb^2 \times \Rb^2\setminus\{0\} $ 
\begin{equation}\label{eq:definition of Jk}
J^q f(x,\xi)= \sum_{i,j=1}^2\int\limits_{-\infty}^\infty t^qf_{ij}(x+t\xi)\,\xi_{i}\xi_{j}  \, d t \quad\mbox{for}\quad(x,\xi)\in \mathbb{R}^{2} \times \mathbb{R}^{2}\setminus \{0\}.
\end{equation} 
It has been shown in~\cite[Equation 2.6]{KMSS} that the data $(I^0 f, I^1 f, I^2 f)$ and $(J^0 f, J^1 f, J^2 f)$ are equivalent for and there are explicit relations between these operators:
\begin{align*}
    	J^0\!f(x,\xi)&=|\xi| I^0\!f
	\left(x-\frac{\langle x, \xi \rangle}{|\xi|^2}\xi,\frac{\xi}{|\xi|}\right)\\
J^1\!f(x,\xi)&=-\frac{\langle\xi,x\rangle}{|\xi|}\,I^0\!f \left(x-\frac{\langle x, \xi \rangle}{|\xi|^2}\xi,\frac{\xi}{|\xi|}\right) + I^1\!f
	\left(x-\frac{\langle x, \xi \rangle}{|\xi|^2}\xi,\frac{\xi}{|\xi|}\right)\\
J^2\!f(x,\xi)&=\frac{\langle\xi,x\rangle^{2}}{|\xi|^3}\,I^0\!f \left(x-\frac{\langle x, \xi \rangle}{|\xi|^2}\xi,\frac{\xi}{|\xi|}\right) - 2\frac{\langle\xi,x\rangle}{|\xi|^2}\,I^1\!f
	\left(x-\frac{\langle x, \xi \rangle}{|\xi|^2}\xi,\frac{\xi}{|\xi|}\right)\\
    &\qquad + |\xi|^{-1}\,I^2\!f
	\left(x-\frac{\langle x, \xi \rangle}{|\xi|^2}\xi,\frac{\xi}{|\xi|}\right).
\end{align*}
% \tred{The entire para will go to the preliminary section. The Fourier transform of a symmetric $2$-tensor field $f \in \Sc(S^2)$ is defined component-wise, that is, 
% \begin{align*}
%  \widehat{f}_{ij}(y)  = \widehat{f}_{ij}(y), \quad y \in \Rb^2 ,
% \end{align*}
% where $\widehat{h}(y)$ denotes the usual Fourier transform of a scalar function $h$ defined on $\Rb^2$.} \vspace{2mm}
% \noindent The Fourier transform  $\Fc :\Sc(T\Sb^1) \longrightarrow \Sc(T\Sb^1)$ is defined as follows, see \cite[Section 2.1]{Sharafutdinov_book}:
% \begin{align}\label{eq:Fourier transform on sphere bundle}
%     \Fc (\vf) (y, \xi) =    \widehat{\vf}(y, \xi) = \frac{1}{(2 \pi)^{1/2}}\int_{\xi^\perp} e^{-i x\cdot y} \vf(x, \xi)\, dx,
% \end{align}
% where $dx$ is the arc-length measure on the line $\xi^\perp = \{ x\in \Rb^2 : \l x,\xi \r = 0\}.$ \vspace{2mm} 
This definition of the Fourier transform is used to compute the following Fourier transform of $q$-th integral moment transform of $f$:
\begin{align*}
    \widehat{I^q f}(y, \xi) = (2 \pi)^{1/2} i^q \langle \xi, \partial_y\rangle^q \left(  \sum_{i,j=1}^2\widehat{f}_{ij}(y) \xi_i\xi_j\right).
\end{align*}
For $q=0$, the above equality reduces to
\begin{align*}
    \widehat{I f}(y, \xi) = (2 \pi)^{1/2}  \sum_{i,j=1}^2 \widehat{f}_{ij}(y) \xi_i\xi_j.
\end{align*}

\begin{theorem}\label{th:injectivity result}
Let $ f \in \mathcal{S}(\mathbb{R}^2;S^2)$ be a symmetric $2$-tensor field then $I^q f = 0$ for  $ q = 0, 1$ if and only if $f = \D^2 v$, for some smooth function $v$  satisfying $v(x) , \partial_{x_i} v \rightarrow 0$ as $|x| \rightarrow  \infty$, for $i=1,2$.
\end{theorem}
\begin{proof}
To prove the `if' part of the theorem, we start by assuming $ f=\d^{2} v$ for some $ v \in C^{\infty}(\Rb^2)$ satisfying   $\d^{\ell}v \rightarrow 0 $ as $ |x| \rightarrow \infty$ for $ \ell=0,1$. Then a simple application of integration by parts  gives
    $$ I^{\ell}(f)= (-1)^{\ell}\,I^0(\d^{2-\ell}v)=0, \qquad \mbox{ for } \ell = 0 , 1.$$ 
Conversely, let $ f\in \mathcal{S}(\mathbb{R}^2;S^{2}) $ be symmetric $2$-tensor field that satisfies $ I^{q}f =0$ for $q=0,1$. Since, $I^0f (x, \xi) = 0$ for $(x, \xi) \in T\Sb^1$. We will use this information to conclude $\int_{\Rb^2} f(x) dx = 0$, so that we can apply our Decomposition Theorem~\ref{th:new_decomposition}. 
\vspace{2mm}\\
Let us fix $\xi = e_1 := (1, 0)$ then for all $x_2 \in \Rb$, we have 
\begin{align*}
    I^0f ((0,x_2), e_1) = 0 \Longrightarrow  \int_\Rb f_{11}(t, x_2)dt = 0 \Longrightarrow  \int_\Rb\int_\Rb f_{11}(t, x_2)dtdx_2 = 0 \Longrightarrow  \int_{\Rb^2}f_{11}(x)dx = 0. 
\end{align*}
Repeating the same process as above by choosing $\xi$ to be unit vectors $(0, 1)$ and $\frac{(1, 1)}{\sqrt{2}}$, we can show $ \int_{\Rb^2}f_{22}(x)dx = 0$  and $  \int_{\Rb^2}f_{12}(x)dx = 0.$
%$$\int_{\Rb^2}f_{22}(x)dx = 0 \qquad \& \qquad  \int_{\Rb^2}f_{12}(x)dx = 0.$$
Therefore, we have $\int_{\Rb^2}f(x)dx = 0$ and hence from our Decomposition Theorem~\ref{th:new_decomposition}, $f$ can be decomposed as $ f=g+\d^{2}v,\quad \delta^{2} g=0 \quad \mbox{and}\quad  v\, \rightarrow 0 \quad \mbox{as} \quad |x| \rightarrow \infty.$
	%\[f=g+\d^{2}v,\quad \delta^{2} g=0 \quad \mbox{and}\quad  v\, \rightarrow 0 \quad \mbox{as} \quad |x| \rightarrow \infty.\]
Using~\cite[Lemma 4.2]{Rohit_Suman}, we have  $J^0 \tilde{f}^i (x, \xi) = 0$ for $ i = 1, 2$ and $\tilde{f}^i = (f_{i1}, f_{i2}),$ since $J^p f(x,\xi) =0$ for $p=0,1$.   
This implies  $ J^0 \tilde{f}^i (x, \xi)|_{T\mathbb{S}^1}= I^0 \tilde{f}^i (x, \xi) = 0, \mbox{ for } i = 1, 2.$
Then, taking the Fourier transform, we get for $y \perp \xi$
\begin{align*}
     \sum_{j=1}^2\widehat{(\tilde{f}^i)}_{j}(y)\xi_j &= 0 \Longrightarrow   \sum_{j=1}^2\widehat{f}_{ij}(y) \xi_j = 0, \quad \mbox{ for } i = 1, 2. 
\end{align*}
This along with decomposition result stated in Theorem~\ref{th:new_decomposition}, we deduce (for $i=1, 2$) that 
\begin{align*}
     \widehat{g}(y) + i_y^2 \widehat{v}(y)  = \widehat{f}(y)
  \implies   \sum_{j=1}^2\xi_j\widehat{g}_{ij}(y) +  \sum_{j=1}^2\xi_j y_iy_j \widehat{v}(y)  =   \sum_{j=1}^2\xi_j \widehat{f}_{ij}(y) \implies  \sum_{j=1}^2\xi_j\widehat{g}_{ij}(y)  = 0. %\  \mbox{ for } i = 1, 2.
\end{align*}
From above, we can obtain two independent relations $\langle 
\widehat{g}, \xi\otimes \xi\rangle = 0$ and $\langle 
\widehat{g}, y \otimes \xi\rangle = 0$. Furthermore $\delta^2 g = 0$ implies  $\langle 
\widehat{g}, y\otimes y \rangle = 0$. Combining all these three conditions, we have $\widehat{g}(y)= 0$ for $y \neq 0$. Since $ \widehat{g}(y) $ is a locally integrable function, therefore as a distribution the support of $ \widehat{g}\subseteq \{0\}$. Amending the arguments used in the proof of the uniqueness part of the Theorem~\ref{th:new_decomposition}, we get $g=0 $ in $ \Rb^2$. Putting $g=0$ in the decomposition above, we achieve $f =\d^{2}v$, which completes the proof of the converse part as well.
\end{proof}

\section{Applications to Elastic ray transform}
\label{sec:Application to elastic ray}

We start with recalling the definition (see~\cite{ilmavirta2025elasticraytransform}) of the \emph{elastic ray transform} acting on the Schwartz class of elastic tensor fields in $\mathcal{S}(\mathbb{R}^n;{E^m(n))}$. For a given direction vector $\xi \in \mathbb{S}^{n-1}$, we denote the hyperplane passing through the origin and perpendicular to $\xi$ by $\xi^\perp$. For a given unit vector $\xi$, the vectors $ \Rb \xi \cup \xi^{\perp} \subset \Rb^n$ are called polarization vectors in $\mathbb{R}^n$. We define
%$
\begin{equation*}
\mathcal{Z}
=
\{(\xi, \zeta): \xi\in \Sb^{n-1} \mbox{ and } \zeta \in \Rb \xi \cup \xi^{\perp}\}
\end{equation*}
%$
to be the set of pairs of unit vectors and corresponding polarization vectors.

\begin{definition}
\label{def: new elastic ray transform}
For $m \in \mathbb{N}$ the elastic ray transform  is the map $X^m \colon \mathcal{S}(\mathbb{R}^n;E^m(n)) \to C^\infty(\mathbb{R}^n \times \mathcal{Z})$ acting on $f \in \mathcal{S}(\mathbb{R}^n;E^m(n))$ by
\begin{equation}
X^mf(x,\xi,\zeta)
=
\int_\mathbb{R}
\inner{f(x+t\xi),(\xi \otimes \zeta)^{\otimes m}}
\,dt
\end{equation}
where $(x,\xi,\zeta) \in \mathbb{R}^n \times \mathcal{Z}$.
\end{definition}
\noindent The aim of this section is to prove injectivity of the elastic ray transform of elastic $2$-tensor fields in $\mathbb{R}^2$ up to the natural obstructions. That is we are going to study the injectivity of the operator $X^2$ on $\mathcal{S}(\mathbb{R}^2;E^2(2))$. The obstruction to uniqueness for $X^2$ come in a form of solenoidal tensor fields adapted to the framework of elastic tensor fields. We will begin by showing that elastic tensor fields in $\Rb^2$ admit specific solenoidal decompositions in . To this end, let us recall certain differential operators from~\cite{ilmavirta2025elasticraytransform}.

With this decomposition result at hand, we state and prove the main result of this section, which is about the kernel description of $X^2$. In dimensions $n\ge 3$, this result was proved in~\cite{ilmavirta2025elasticraytransform}; however, the exact kernel description for $n=2$ was not proved in~\cite{ilmavirta2025elasticraytransform} because their proof of the generalized trace-free solenoidal potential decomposition is not applicable in that case.
It is, however, possible to prove the decomposition result under an additional mean zero condition $\int_{\Rb^2} f_{ijk\ell}(x) \,\d x=0$ when $f\in \sch(\mathbb{R}^2;E^2(2))$ as we show.

\begin{theorem}
\label{thm:elastic-injectivity}
For $f \in \mathcal{S}(\mathbb{R}^2;E^2(2))$ the following are equivalent:
\begin{enumerate}[label=(\alph*)]
    \item  \label{item:vanishing-elastic-ray-transform} The elastic ray transform of $f$ vanishes.
    \item \label{item:elastic-potential-form} There are $v \in C^\infty(\mathbb{R}^2;E^1(2))$ and $u \in C^\infty(\mathbb{R}^2;\mathbb{R}^2)$ so that $u$ and $v$ and all their partial derivatives tend to $0$ as $\abs{x} \to 0$, so that $f = Hv + Ku$.
\end{enumerate}
\end{theorem}

\begin{proof}
Assume first that~\ref{item:elastic-potential-form} is true i.e. there are $v \in C^\infty(\mathbb{R}^2;E^1(2))$ and $u \in C^\infty(\mathbb{R}^2;\mathbb{R}^2)$ so that $u$ and $v$ and all their partial derivatives tend to $0$ as $\abs{x} \to 0$, so that $f = Hv + Ku$. For any $(\xi,\zeta) \in \mathcal{Z}$ and $x \in \mathbb{R}^2$ we have 
\begin{equation}
\begin{split}
X^2(Hv)(x,\xi,\zeta)
&=
\int_{\mathbb{R}}
\sum_{i,j,k,l}
\frac{\partial^2v_{kl}}{\partial x_i\partial x_j}
(x + t\xi)
\xi_i\zeta_j\xi_k\zeta_l
\,\d t=
\int_{\mathbb{R}}
\sum_{j,k,l}
\partial_t
\left(
\frac{\partial v_{kl}}{\partial x_j}
(x + t\xi)
\zeta_j\xi_k\zeta_l
\right)
\,\d t=
0,
\end{split}
\end{equation}
since all partial derivatives of $v$ tend to $0$ as $\abs{x} \to \infty$. On the other hand, for any $(\xi,\zeta) \in \mathcal{Z}$ and $x \in \mathbb{R}^2$ we have
\begin{equation}
X^2(Ku)(x,\xi,\zeta)
=
\inner{\xi,\zeta}
\frac{1}{2}
\int_{\mathbb{R}}
\sum_{i,j}
\left(
\frac{\partial u_j}{\partial x_i}(x + t\xi)
+
\frac{\partial u_i}{\partial x_j}(x + t\xi)
\right)
\xi_i\zeta_j
\,\d t
\end{equation}
If $\xi \perp \zeta$, then $X^2(Ku)(x,\xi,\zeta) = 0$. Otherwise $\zeta = \lambda \xi$ for some $\lambda \in \mathbb{R}$ in which case
\begin{equation}
\inner{\xi,\zeta}
\frac{1}{2}
\int_{\mathbb{R}}
\sum_{i,j}
\left(
\frac{\partial u_j}{\partial x_i}(x + t\xi)
+
\frac{\partial u_i}{\partial x_j}(x + t\xi)
\right)
\xi_i\zeta_j
\,\d t
=
C
\int_{\mathbb{R}}
\sum_i
\partial_t( u_i (x + t) \xi_i) \,\d t
=
0
\end{equation}
since $u$ tends to $0$ as $\abs{x} \to 0$. Together the fact that $X^2(Hv) = 0$ and $X^2(Ku) = 0$ show that $X^2f = 0$.

Conversely, let us assume that~\ref{item:vanishing-elastic-ray-transform} is true i.e. $X^2f = 0$. Then the Fourier slice theorem implies that $f$ has mean zero i.e. $\int_{\mathbb{R}^2}f(x)\,\d x = 0$ (see~\cite[Lemma 7]{ilmavirta2025elasticraytransform}). It follows from Theorem~\ref{th:decomposition_elastic} that there are $u \in C^\infty(\mathbb{R}^2;\mathbb{R}^2)$ and $v \in C^\infty(\mathbb{R}^2;E^1(2))$ and $g \in C^\infty(\mathbb{R}^2;E^2(2))$ so that $f = Hv + Ku + g$ with $H^*g = K^*g = 0$ and $\abs{\partial_kv_{ij}(x)} \leq C(1 + \abs{x})^{-1}$ and $\abs{u_i(x)} \leq C(1 + \abs{x})^{-1}$. By our computation earlier in this proof, we see that
\begin{equation}
\label{eqn:constraint-on-g}
0
=
X^2f
=
X^2g.
\end{equation}
It was shown in~\cite{ilmavirta2025elasticraytransform} that~\eqref{eqn:constraint-on-g} together with $H^*g = K^*g = 0$ force $g$ to vanish identically. Hence $f = Hv + Ku$ and the proof is complete.
\end{proof}

\begin{remark}
The solenoidal injectivity result of Theorem~\ref{thm:elastic-injectivity} can be extended to $L^2$ integrable elastic tensor field in $2$ dimensions using the techniques introduced in~\cite{ilmavirta2025elasticraytransform}. An $L^2$ result is given there for dimensions $n \geq 3$ and the only missing pieces to prove the result in $2$ dimensions are the decomposition results in Lemma~\ref{lem:E2-decomposition} and Theorem~\ref{th:decomposition_elastic}. We decided not to rerecord these results or repeat their proofs in $2$ dimensions, since after showing the decompositions of elastic tensor fields the proof can be completed exactly as in~\cite{ilmavirta2025elasticraytransform}.
\end{remark}

\section{A connection between the momentum and elastic ray transforms}
\label{sec:relation between integral transform}

In this section we show that the kernel of the integral moment transform $\{I^0,I^1\}$ is the same as the kernel of the elastic ray transform $X^2$. Recall, that in coordinates, the elastic ray transform of elastic $1$-tensor field $f \in \sch(\mathbb{R}^n;E^1(n))$ is
\begin{align}
X^1 f(x, \xi, \zeta)
=  \sum_{i,j=1}^2 \int_{\mathbb{R}} f_{ij}(x + t \xi)\xi_i \zeta_j\quad\text{for}\quad\zeta \in \Rb \xi \cup \xi^{\perp}.
\end{align}
An elastic $1$-tensor field is simply a symmetric $2$-tensor field which allows us to package the information of $X^1f$ as $\{I^0f,Mf\}$ where $M$ is the mixed ray transform defined by
\begin{align}
    Mf(x, \xi, \zeta)=  \sum_{i,j=1}^2\int_{\mathbb{R}} f_{ij}(x + t \xi)\xi_i \zeta_j \quad\text{for}\quad \zeta\perp \xi.
\end{align}
We can write any $\zeta \perp \xi$ as
\begin{equation}
\zeta(\eta)
=
P_\xi\eta
=
\eta
-
\inner{\xi,\eta}\xi \qum{for $\eta\in\Rt\setminus\{0\}$.}
\end{equation}
Then $Mf$ can be written as
\begin{align}
Mf(x,\xi,\zeta(\eta))
= \sum_{i,j=1}^2 \int_{\mathbb{R}} f_{ij}(x + t \xi)\xi_i \zeta_j(\eta).
\end{align}
With this notation at hand, we have the following result. We freely identify elastic $1$-tensor fields with symmetric $2$-tensor fields in the statement.

\begin{proposition}
\label{prop_equi}
As subsets of $\sch(\mathbb{R}^n;S^2)$ the kernels of the transforms $\{I^0,I^1\}$ and $X^1$ are equal.
\end{proposition}

\begin{proof}
Recall the operators $J^0$ and $J^1$ defined~\eqref{eq:definition of Jk}, which are extensions of $I^0$ and $I^1$, respectively, to $\Rb^n \times \Rb^n \setminus \{0\}$. A straightforward computation shows that
\begin{align}\label{eq_16}
	\p_{\xi_j} J^0f(x, \xi) - \p_{x_j}J^1f (x, \xi) = 2 J^0f_j(x, \xi) =  2\sum_{i=1}^2\int_{\Rb} f_{ij}(x+t\xi)\, \xi_i d t. 
\end{align}
Contracting the above relation by $\zeta \in \xi^{\perp}$ we get
\begin{align}
 \sum_{j=1}^2\zeta_j\lr{\p_{\xi_j} J^0f(x, \xi)  - \p_{x_j}J^1f(x, \xi)} =  2\, Mf(x, \xi, \zeta), 
\end{align}
where $Mf (x, \xi, \zeta)$ is the mixed ray transform of $f$. Therefore, if $\{I^0f,I^1f\} = 0$, then also $X^1f = \{I^0f,Mf\} = 0$.
\vspace{2mm}\\
Conversely, assume that $X^1f = \{I^0f,Mf\} = 0$. We compute that
\begin{align}
\partial_{\eta_k} Mf (x, \xi, \zeta(\eta))
&= \sum_{i,j=1}^2 \int_{\Rb} f_{ij}(x+t\xi)\xi_i \partial_{\eta_k}\zeta_j(\eta)=  \sum_{i,j=1}^2 \int_\mathbb{R}
f_{ij}(x+t\xi)\xi_i(\delta_{kj} + \xi_k\xi_j)
\\ &=  \sum_{i= 1}^2 \int_{\Rb} f_{ik} (x+t\xi)\xi_i +
\xi_k  \sum_{i,j=1}^2 \int_{\Rb} f_{ij} (x+t\xi)\xi_i \xi_j .
\end{align}
Using~\eqref{eq_16} we obtain
\begin{equation}
\p_{\eta_k}Mf (x, \xi,\zeta(\eta))
=
\frac{1}{2}
\left(
\p_{\xi_k} J^0f (x, \xi)
-
\p_{x_k}J^1f (x, \xi)
\right)
+
\xi_kJ^0f(x,\xi).
\end{equation}
Since $\{I^0f,Mf\} = 0$, the above equation implies that $\partial_{x_k}J^1f(x,\xi) = 0$ for all $k$ and $(x,\xi)$. That is, $J^1f(x,\xi)$ is constant for all $(x,\xi)$. Let this constant be $C \in \mathbb{R}$. Then using any $\xi' \in \xi^\perp$ and $J^1f|_{T\mathbb{S}^1} = I^1f \in \sch(T\mathbb{S}^1)$ we find that
\begin{equation}
C
=
\lim_{s \to \infty}
J^1f(x+s\xi',\xi)
=
0
\end{equation}
Thus we have shown $\{I^0f,I^1f\} = 0$ proving the equality of the kernels of the two transforms.
\end{proof}
% \begin{definition}\label{def:deneralized Snt _Venant_ope}
% The generalized Saint Venant operator $ \Rc^1:C^\infty(\Rb^n;S^m)\rightarrow C^\infty(S^{m-k} \otimes S^m) $ is defined by the equality 
% \begin{align}\label{eq:Genralized Saint-Venant operator}
% (\Rc^1 f)_{ijk}= \sigma(j,k)\left\{ \frac{\partial f_{ij}}{\partial x_k}  - \frac{\partial f_{ik}}{\partial x_j}  \right\}
% \end{align}
% \end{definition}
\noindent As a corollary, we obtain the following kernel characterization for the elastic ray transform of elastic $1$-tensor fields.
\begin{corollary}
Let $f\in \sch(\mathbb{R}^n;E^1(n))$. Then the following conditions are equivalent.
\begin{enumerate}
\item The elastic ray transform $X^1f$ of $f$ vanishes.

\item There are $\phi \in C^\infty(\mathbb{R}^n)$ such that $f= \d^2 \phi$,  with $\phi$ and $\partial_{x_i}\phi$ tend to zero as $\abs{x} \to \infty$ for all $i$.

\item The tensor field $f$ is in the kernel of the generalized Saint-Venant operator $\mathcal{R}^1$ defined by
\begin{equation}
(\mathcal{R}^1f)_{ijk}(x)
=
\frac{1}{2}
\left(
\frac{\partial f_{ij}}{\partial x_k}(x)
-
\frac{\partial f_{ik}}{\partial x_j}(x)
\right).
\end{equation}
\end{enumerate}
\end{corollary}
\begin{proof}
By Proposition~\ref{prop_equi}, we have that $X^1f=0$ if and only if $I^0f=I^1f=0$. Hence, the result follows from Theorem~\ref{th:injectivity result} for $n = 2$ and from~\cite[Theorem 3.1]{MS2023} and~\cite[Theorem 4.3]{Rohit_Suman} for $n \geq 3$.
\end{proof}
\appendix
\section{Homogeneous distributions}\label{appen_homo_dis}
%\ssc{I will modify this little bit.}
We start with recalling some of the results from~\cite{Dorina,Hormander_I,Taylor_book}.
\begin{definition}\label{def_distribution}
Let $\Omega$ be an open subset of $\mathbb{R}^n$. A linear functional 
$u : C_c^\infty(\Omega) \to \mathbb{R}$ 
is called a \emph{distribution} if for every compact set 
$K \subset \Omega$, there exist a constant $C \ge 0$ 
and a non-negative integer $N$ such that
\begin{align}\label{eq_1.1}
	|(u, \phi)| \le C \sum_{|\alpha|\le N} \sup_{K} | \partial^{\alpha} \phi | 
	\quad \text{for all } \phi \in C_c^\infty(\Omega) \text{ with } \mathrm{supp}\, \phi \subset K.
\end{align}
The inequalities in \eqref{eq_1.1} are called \emph{seminorm estimates}.
\end{definition}
\noindent Let $f \in L^1_{\mathrm{loc}}(\Omega)$. The distribution associated with $f$ is denoted by $T_f$ and defined as 
\begin{equation}\label{def_T_f}
	T_f(\phi) = \int_{\Omega} f(x)\, \phi(x) \, \mathrm{d}x 
	= \int_{\Omega} f \phi.
\end{equation}
It is easy to see that 
\begin{equation}
	|T_f(\phi)| \le C \sup_{K} |\phi| 
	\quad \text{for all } \phi \text{ with } \mathrm{supp}\, \phi \subset K,
\end{equation}
which shows that $T_f$ is indeed a distribution.

Let $f(x) = |x|^{-a}$, $x \in \mathbb{R}^n$, $a \in \mathbb{R}$. 
Then $f \in L^1_{\mathrm{loc}}(\mathbb{R}^n)$ if and only if $a < n$. 
Therefore, $|x|^{-a}$ defines a distribution whenever $a < n$. 
The natural question is: what happens when $a = n$?

We start with the simplest case $n = a = 1$. 
Recall the Cauchy principal value distribution $\mathrm{p.v.}\,\frac{1}{x}$. 
It acts on a test function $\phi(x) \in C_c^\infty(\mathbb{R})$ by
\[
\langle \mathrm{p.v.}\, \tfrac{1}{x}, \phi \rangle 
= \lim_{\epsilon \to 0^+} \int_{|x| > \epsilon} \frac{\phi(x)}{x} \, \mathrm{d}x.
\]
Using the fact that $\frac{1}{x}$ is an odd function, we can rewrite this as
\begin{align*}
    \langle \mathrm{p.v.}\, \tfrac{1}{x}, \phi \rangle 
    &= \lim_{\epsilon \to 0^+} \int_{|x| > \epsilon} \frac{\phi(x) - \phi(0)}{x} \, \mathrm{d}x= \int_0^{\infty} \frac{\phi(x) - \phi(-x)}{x}\, \mathrm{d}x.
\end{align*}
One can define $\mathrm{p.v.}\,\frac{1}{x^2}$ similarly.

Next, we extend this idea to higher dimensions. To this end, recall that 
$\mathcal{S}(\mathbb{R}^n)$ denotes the Schwartz space of rapidly decreasing functions. 
For $f \in \mathcal{S}(\mathbb{R}^n)$, we have
\begin{align}
   \sup_{x \in \mathbb{R}^n} |x^{\alpha} \partial^{\beta} f(x)| < \infty 
   \quad \text{for all multi-indices } \alpha, \beta.
\end{align}
We define $\mathcal{S}'(\mathbb{R}^n)$ as the space of bounded linear functionals 
from $\mathcal{S}(\mathbb{R}^n)$ to $\mathbb{R}$. 
An element $u \in \mathcal{S}'(\mathbb{R}^n)$ is called a 
\emph{tempered distribution} if there exist non-negative integers $m,k$ 
and a constant $C > 0$ such that 
\begin{align}
|(u, \phi)| \le C \sum_{\substack{x \in \mathbb{R}^n \\ |\alpha|\le m,\, |\beta|\le k}} 
|x^{\alpha} \partial^{\beta} \phi(x)|, 
\quad \forall \phi \in \mathcal{S}(\mathbb{R}^n).
\end{align}
The advantage of working with $\mathcal{S}(\mathbb{R}^n)$ and 
$\mathcal{S}'(\mathbb{R}^n)$ is the following:
\begin{itemize}
    \item[1.] The Fourier transform $\mathcal{F}: \mathcal{S}(\mathbb{R}^n) \to \mathcal{S}(\mathbb{R}^n)$ 
    is bijective and continuous, and its inverse is also continuous.
    \item[2.] The Fourier transform $\mathcal{F}: \mathcal{S}'(\mathbb{R}^n) \to \mathcal{S}'(\mathbb{R}^n)$ 
    is bijective and continuous, and its inverse is also continuous.
\end{itemize}

For $u \in \mathcal{S}'(\mathbb{R}^n)$, its Fourier transform is defined by
\begin{align}
    (\mathcal{F}u, \phi) = (u, \mathcal{F}\phi) \quad \forall \phi \in \mathcal{S}(\mathbb{R}^n).
\end{align}
Examples of tempered distributions include $|x|^{-a}$ for $a < n$.  
We now define the tempered distribution associated with $|x|^{-n}$ in $\mathbb{R}^n$ 
and compute its Fourier transform.
\vspace{2mm}\\
\noindent Recall the Cauchy principal value distribution 
$\mathrm{p.v.}\,f(x) \in \mathcal{S}'(\mathbb{R})$, where $f(x) = \frac{1}{x}$. 
Note that 
\begin{itemize}
    \item $f \in C^{\infty}(\mathbb{R} \setminus \{0\})$,
    \item $f$ is positively homogeneous of degree $-1$,
    \item $f$ is odd, i.e. $f(1) + f(-1) = 0$.
\end{itemize}
Extending these properties to $\mathbb{R}^n$, we now define the 
Cauchy principal value distribution in $\mathbb{R}^n$. 
\vspace{2mm}\\
Let $f\in C^{\infty}(\Rn\setminus\{0\})$, homogeneous of degree $-n$ and $\int_{\sn} f\, \d \omega=0$. Then
\begin{equation}\label{def_pv}
(\mathrm{p.v.}(f),\phi) = \int_{\Rn} f(x)( \phi(x)- \phi(0) \psi(x)) \d x \quad \forall \phi\in \sch(\mathbb{R}^n),
\end{equation}
where
% \begin{align}\label{def_of_psi}
%     \mbox{$\psi$ is a radial  $C^{\infty}$ function with $\psi(0)=1$ and $\abs{\psi}$ behaves like $\abs{x}^{-\delta}$ for $\abs{x}$ large and for some  $\delta\in (0, \infty)$.}
% \end{align}
\begin{itemize}\label{def_of_psi}
    \item $\psi$ is a radial  $C^{\infty}$ function with $\psi(0)=1$,
    \item $\abs{\psi}$ behaves like $\abs{x}^{-\delta}$ for $\abs{x}$ large and for some  $\delta\in (0, \infty)$.
\end{itemize} 
The above definition is independent of 
 the choice of $\psi$ due to $\int_{\sn} f\, \d \omega=0$. This defines a tempered distribution in $\Rn$. However, when  $f(x) = \abs{x}^{-n}$ the  condition $\int_{\sn} f\, \d \omega=0$ fails. However, one could still work with the above definition due to following reason:
\begin{align*}
 ( \phi(x)- \phi(0) \psi(x))= x\Tilde{\phi} \quad \mbox{for some $\Tilde{\phi}\in \sch(\Rn)$},   
\end{align*}
since this entails $f(x)( \phi(x)- \phi(0) \psi(x))\in L^1(\Rn)$. But this distribution is going to  depends on the choice of  $\psi$.
To this end, we  denote \[T_{\psi}f =\mathrm{p.v.}\frac{1}{\abs{x}^n} \quad \mbox{and}\quad T_{\psi}f\in \srn.\]
Its action on $\sch(\Rn)$ is given by 
\begin{align}
    \left(\mathrm{p.v.}\frac{1}{\abs{x}^n}, \phi\right) = \int_{\Rn} \frac{\phi(x) -\psi(0) \psi(x)}{\abs{x}^n} \quad \forall \phi \in \sch(\Rn),
\end{align}
where $\psi$  is a radial  $C^{\infty}$ function with $\psi(0)=1$ and $\abs{\psi}$ behaves like $\abs{x}^{-\delta}$ for $\abs{x}$ large and for some  $\delta\in (0, \infty)$. One could take $ \psi(x) = e^{-\abs{x}^2/2}.$

% Next we denote the dilation operator as $\tau_{\lambda} F(x) =F(\lambda x)$, for $\lambda\in \R$.  Let $u\in \srn$ then we have 
% \begin{align*}
%     (\tau_{\lambda} u, \phi) = \lambda^{-n}(u,\tau_{\lambda^{-1}} \phi) \quad \forall \phi\in \Sch.
%     \end{align*}
%     Moreover $\tau_{\lambda} u\in \srn$.

\section{Acknowledgments}\label{sec:acknowledgement}

A.K. acknowledges the support of the Geo-Mathematical Imaging Group at Rice University. SKS is supported by IIT Bombay
seed grant (RD/0524-IRCCSH0-021) and ANRF Early Career Research Grant (ECRG) (RD/0125-ANRF000-016). RKM is supported by ANRF-Matrics grant ANRF/ARGM/2025/000468/MTR.

\bibliographystyle{plain}
\bibliography{ref}	
\end{document}